\documentclass[11pt,reqno,makeidx]{amsart}
\usepackage[active]{srcltx}
\usepackage{amssymb,color}
\usepackage{float}
\restylefloat{table}

\usepackage{ragged2e}
\usepackage{array}
\usepackage[active]{srcltx}
\usepackage{epsfig}
\usepackage{hyperref}
\usepackage{amsmath}
\usepackage{amssymb}
\usepackage{hyperref}
\usepackage{enumerate}
\newtheorem{Proposition}{Proposition}
  \newtheorem{Remark}[Proposition]{Remark}
  \newtheorem{Corollary}[Proposition]{Corollary}
  \newtheorem{Lemma}[Proposition]{Lemma}
  
  \newtheorem{Theorem}{Theorem}
 
 \newtheorem{Definition}[Proposition]{Definition}
 \newtheorem{Note}[Proposition]{Note}

\newcommand {\z}{{\noindent}}

\def\Le{\leqslant}
\def\Ge{\geqslant}
 \def\HH{\mathbb{H}}
\def\CC{\mathbb{C}}
\def\DD{\mathbb{D}}
 \def\RR{\mathbb{R}}
 \def\NN{\mathbb{N}}
\def\ZZ{\mathbb{Z}}

\def\Re{\mathrm{Re\,}}
\def\Im{\mathrm{Im\,}}

\def\JN{ J}
\def\LN{ L}
\def\newD6p{\mathbb{D}_{21/4}^+}
 \def\({\left(} \def\){\right)} \makeindex
\def\ro{{r}}
\def\tcT{{2,3}}
\def\tsT{{2;3}}

\def\tso{{2;1}}
\def\Tco{{3,1}}

\def\sm{{s^{\scriptscriptstyle -}}}
\def\um{{u^{\scriptscriptstyle -}}}
\def\stjm{{\tilde{s}_j}{\negthinspace}^{\scriptscriptstyle -}}
\def\h,{\negthinspace}

\def\z{\noindent}
\def\be{\begin{equation}}
\def\bel{\begin{equation}\label}
\def\ee{\end{equation}}

\def\bd{\begin{Definition}}
\def\ed{\end{Definition}}

\def\bp{\begin{Proposition}}
\def\bpl{\begin{Proposition}\label}
\def\ep{\end{Proposition}}

\def\bl{\begin{Lemma}}
\def\el{\end{Lemma}}
\def\bt{\begin{Theorem}}
\def\et{\end{Theorem}}

\def\Bx{\ $\Box$}
\def\oldCa{\ell}
\def\oldCb{\ell^-}
\def\oldDa{c_1}
\def\oldC0{c_2}
\def\oc1{c_3}
\def\oldc2{c_4}
\def\ogc1{c_5}
\def\olgoc2{c_6}
\def\olhatl{\tilde{l}}
\def\olhatg{\tilde{g}}

\author{O. Costin, R.D. Costin and M. Huang} \address{Mathematics
  Department\\The Ohio State University\\Columbus, OH 43210}
\address{Mathematics Department\\The Ohio State University\\Columbus, OH
  43210} \address{Mathematics Department\\City University of Hong Kong, Hong Kong
} \title%{New method to find connection formulae for $\text{P}_1$}
{A direct method to find Stokes multipliers in closed form for
P1 and more general integrable systems}
\begin{document}
$ $ \vskip -2cm
\gdef\shorttitle{Direct method to find Stokes multipliers in closed form}
\maketitle

%\newpage
%\tableofcontents

\begin{abstract}
We calculate the Stokes multipliers in closed form for
P1 using a direct approach. For this purpose, we introduce a new rigorous method, based on Borel summability and
  asymptotic constants of motion generalizing previous results,
  to analyze singular behavior of nonlinear ODEs in a neighborhood of infinity
  and provide global information about their solutions in $\CC$. In equations with the
  Painlev\'e-Kowalevski (P-K) property (stating that movable singularities are
  not branched) the method allows for solving connection problems.  The analysis is
  carried in detail for P$_1$, $y''=6y^2+z$, for which we find the Stokes
  multipliers in closed form and global asymptotics in $\CC$ for solutions having
  power-like behavior in some direction, in particular for the
  tritronqu\'ees.

  Calculating the Stokes multipliers solely relies on the P-K property and
  does not use linearization techniques such as Riemann-Hilbert or
  isomonodromic reformulations.  %We discuss how the approach would work to
 % calculate connection constants for a larger class of P-K integrable equations.

  We develop methods for finding asymptotic expansions in sectors where
  solutions have infinitely many singularities. These techniques do not rely
  on integrability and apply to more general second order ODEs which, after
  normalization, are asymptotically close to autonomous Hamiltonian systems.
\end{abstract}

\section{Introduction}\label{out}

\subsection{Overview of the paper and motivation}

We provide a new method for studying solutions of nonlinear second
order equations in singular regions, containing singularities which
may accumulate towards infinity.  The method relies on obtaining and
using asymptotically conserved quantities (ACQ); these may not exist
globally, but rather on regions bordered by antistokes lines. These ACQ
can be matched to each other and to ACQ valid in regular regions (bordered by antistokes lines, where solutions are analytic towards infinity). A family of matched ACQ
determines solutions and their behavior, and the use of ACQ as dependent
variables desingularizes the problem.

For analytic ODEs it is known that solutions which are
regular in sectors toward infinity have
asymptotic expansions which are Borel summable \cite{imrn,duke,
  invent}. These expansions are shown here to provide ACQ that match the ACQ of the singular regions. This  approach does not use linearization
such as a Riemann-Hilbert reformulation. For equations with the
Painev\'e-Kowalevski property (P-K) -- stating that all solutions are
single-valued on a {\em common} Riemann surface -- the asymptotically
conserved quantities provide explicit connection formulas using the
P-K property alone due to the fact
that the solution returns to the same asymptotic representation after
a $2\pi$ rotation in a neighborhood of infinity. This yields a
nontrivial equation for the {\em Stokes multiplier}, see \S\ref{MainS}.

In the present paper we carry out this program for the Painlev\'e
equation P$_1$
\begin{equation}\label{p1}
y''= 6y^2 + z
\end{equation}
%We discuss in \S\ref{secp2} how the approach can be applied to the Painlev\'e equation P$_2$.

We obtain the asymptotic behavior of tronque\'e solutions with
exponential accuracy in the pole-free sectors, with $O(z^{-\frac{25}{16}})$ relative
errors in pole regions where the associated asymptotic elliptic functions become
trigonometric ones, and $O(z^{-\frac{15}{8}})$ where the elliptic functions
are nondegenerate,  see Theorem\,\ref{P1p}, \S\ref{sec1} and
\S\ref{tro1} below. This precision  exceeds the one needed to keep track
of the Stokes multipliers and then to determine them based on the single-valued consistency mentioned above.

Until now the Stokes multiplier has been calculated using linearization methods;
one interest in developing an alternative approach
is that while there is no known systematic method of generating an associated
Riemann-Hilbert problem from the P-K property, establishing the P-K property is generally a much easier problem.

\subsection{A brief overview of integrability, linearization, the R-H problem and  connection formulas}\label{IlRc}

For more details we refer the reader to the surveys by Clarkson \cite{Clarkson}  and   Joshi \cite{Joshi}.
Eq. \eqref{p1} is the first of the six Painlev\'e equations. These, together
with equations reducible to equations of classical functions,
constitute  the complete set of differential
equations of degree at most two, in a quite general class \cite{Ince}, that are P-K integrable
\cite{Boutroux1,Boutroux2,Clarkson}, modulo  equivalences.

In the realm of linear differential equations, the classical special functions such
as the Airy, Bessel or hypergeometric ones, play an important role
due the existence of integral representations, allowing in particular
for a global description: an integral formula allows for explicitly
linking the behavior of one solution at various critical points, and
along different directions at infinity. These links are {\em
  connection problems} and their solutions are {\em connection
  formulae}. Until the late 1970s integral formulas were essentially
the only tools in solving connection problems. The general solution of \eqref{p1} is highly transcendental; in
particular there are no integral representations in terms of simpler
functions.  This is a deep result with a long history starting with
a partial proof by Painlev\'e himself, and a complete argument due to Umemura,
\cite{Umemura2}, \cite{Umemura3}.

However, all six Painlev\'e equations turned out to have explicit
connection formulas. These have been obtained over a span of about two decades
starting in the late 70s, after the discovery of linearization methods, cf. the fundamental papers
by Ablowitz and Segur \cite{Ablowitz}, McCoy, Wu and Tracy \cite{McCoy}; for a good survey of the vast literature see the papers by Clarkson \cite{Clarkson} and by  Fokas,  Its,  Kapaev  and Novokshenov  \cite{Fokas}. Linearization techniques fall in some sense under the Riemann-Hilbert
(RH) reformulation umbrella, \cite{Fokas}. The six
Painlev\'e equations were shown in the late last century to be RH
compatibility equations, and the P-K property follows from this
presentation.

For the Painlev\'e transcendent $\text{P}_1$ the Stokes multiplier was
obtained via linearization in 1988 by Kapaev \cite{Kapaev}, corrected in a
1993 paper by Kitaev-Kapaev \cite{KitaevKapaev}. The Painlev\'e transcendents
are now as important in nonlinear mathematical physics as the classical
special functions are in linear mathematical physics \cite{Fokas}.

Conversely, knowing the jumps across the cuts determines the associated
RH problem. For a differential equation the jumps follow from  its connection formulas. The circle is closed once the P-K property
is shown to determine the jump conditions in closed form.

\subsection{Definitions,  setting and general properties of P$_1$}\label{sec1}
\subsubsection{Normalization}
It is convenient to normalize \eqref{p1} as described in \cite{invent}.
The change of variables
\begin{equation}
  \label{chvarx}
  z={24}^{-1}{30^{4/5}}x^{4/5}e^{-\pi i /5};\ y(z)=i\sqrt{z/6}(1-\tfrac{4}{25}x^{-2}+h(x))
\end{equation}
(the branch of the square root is positive for
$z>0$) brings \eqref{p1} to the Boutroux-like form
\begin{equation}
  \label{eq:eqp}
h''+\frac{h'}{x}-h-\frac{h^2}{2}-\frac{392}{625}\,\frac{1}{ x^4}=0
\end{equation}
\subsubsection{Symmetries}\label{sym}
Eq.  \eqref{p1} has a five-fold symmetry: if $y(z)$ solves \eqref{p1}, then so
does  $\rho^2y(\rho z)$ if $\rho^5=1$.  Relatedly,
\eqref{eq:eqp} is invariant under the transformations $h(x)\mapsto h(xe^{\pm
      i\pi})$
and note also the symmetry $h(x)\mapsto\overline{h(\overline{x})}$.

\subsubsection{Regularity}\label{tro1} There are five special directions of \eqref{p1} for solutions having asymptotic power series in some sectors (see Note \ref{StokesAz} and \S\ref{Overmu}). Bordered by these directions, we have the sectors
\begin{equation}
  \label{eq:sectorsz}
  S_k=\left\{z\in\CC\, \Big|\, \frac{2k-1}{5}\pi
<\arg z<\frac{2 k+1}{5}\pi \right\},\ k\in\ZZ_5
\end{equation}

{\em Tronqu\'ees and tritronqu\'ees solutions}. Generic solutions
have poles accumulating at $\infty$ in all $S_k$.
Any solution has poles in at least one $S_k$ \cite{Kitaev}. For any
{\em two} adjacent sectors $S_k$ there is a one-parameter family of
solutions, called {\em{tronqu\'ees}} solutions, with the behavior
$y=\pm i\sqrt{\frac{z}{6}}(1+o(1))$ as $z\to\infty$ in both
sectors (so they do not have poles for large $z$ in two sectors). In particular, for any set of {\em four} adjacent sectors there is
exactly one solution with this behavior, see \cite{Kapaev-2004}, \cite{Masoero};
these particular tronqu\'ees solutions which are maximally regular
solutions are called {\em tritronqu\'ees}. The five tritronqu\'ees
are obtained from each other via the five-fold symmetry.

We will study the tritronque\'e  $y_t$ with
\begin{equation}\label{sectorinz}
y_t(z)=i\sqrt{\frac{z}{6}}\,\left(1+o(1)\right)\ \text{as $|z|\to \infty$ with }\arg z\in
  \left(-\frac{3\pi}{5},\pi\right)
\end{equation}
In the  normalization \eqref{eq:eqp}, a sector $S_k$  in $z$ corresponds to a quadrant
in $x$  and the sector $-\pi<{\rm{arg}}\ z\leq \pi$ corresponds to the sector $-\pi<{\rm{arg}}\ x\leq 3\pi/2$. The solution $h_t$ of \eqref{eq:eqp}
corresponding to $y_t$ satisfies
\begin{equation}\label{sectorinx}
h_t(x)=o(1)\ \text{as $x\to \infty$ with }\arg x\in \left[-\frac{\pi}{2},\frac{3\pi}{2}\right]
\end{equation}
 is analytic for
large $x$ in the sector \eqref{sectorinx} and has arrays of  poles beyond its
edges,
see \S\ref{firstpoles}.
General results about the solutions of \eqref{eq:eqp} which decay in some
direction at infinity, which correspond to tronqu\'ee of \eqref{p1},
are overviewed in \cite{cch}.

\section{Main results in sectors of analyticity}\label{MainS}
\subsection{The Stokes constants for  tronqu\'ee solutions}
Theorem\,\ref{P1p} gives the value of the {\em Stokes
   multiplier}  $\mu$ for any tronqu\'ee solution. While it is formulated for
 solutions analytic in the sector $S_0$, it can be easily  adapted  to
 solutions analytic  in any of the sectors in \eqref{eq:sectorsz}.
\begin{Theorem}\label{P1p} Let $h$ be a solution of  \eqref{eq:eqp} satisfying
  \begin{equation}
  \label{eq:eqsm}
  h(x)=Ce^{-x}x^{-\frac12}+o(x^{-1/2}) \ \ \text{ as } \ x\to e^{i\alpha}\infty\ \text{for all}\  \alpha\in [0,\tfrac12 \pi],\ \ \ \ \
\end{equation}
Then
\begin{equation}
  \label{eq:eqsm2}
 h(x)
  =(C+\mu) x^{-\frac12}e^{-x}+o(x^{-\frac12}) \ \ \text{ as }\  x\to e^{i\alpha}\infty\   \text{for all}\   \alpha\in [-\tfrac12 \pi,0] ,\ \ \ \ \
  \end{equation}
with
\begin{equation}
  \label{val_mu}
    \mu=\sqrt{\frac{6}{5 \pi }}\, i
\end{equation}
\end{Theorem}

Of course, for $ |\alpha|\in [0,\tfrac12 \pi)$ relation \eqref{eq:eqsm} simply means that $h(x)=o(x^{-1/2})$, trivially true since in this region $h\sim -\frac{392}{625 x^4}$.

 The {\em existence} a Stokes multiplier such that \eqref{eq:eqsm} and
\eqref{eq:eqsm2} hold is known in a wide class of differential equations, see
in \cite{imrn} formula following (1.15) and also see in \cite{duke},
\eqref{eq:sjump}. A complete Borel summed expansion of the solutions $h$ satisfying (\ref{eq:eqsm}), (\ref{eq:eqsm2}) is given in  \cite{cch}-- Theorem 2 and (55)-- where $C_+=C,C_-=C+\mu$.

On the other hand, the existence of an {\em{explicit expression}} for $\mu$ is expected only in special cases such as
integrable equations. The value \eqref{val_mu} was calculated before using Riemann-Hilbert
associated problems, as mentioned in \S\ref{IlRc}.
 In the present paper,  \eqref{val_mu} will follow from more general asymptotic formulas we obtain by
 matching Borel summed expansions valid in the regular sector to asymptotic
 constants of motion explained in \S\ref{KAM}, which are shown to give suitable representations in the sectors with singularities.

In particular, for the tritronqu\'ee $h_t$ obtained from $y_t$ via \eqref{chvarx}, Theorem\,\ref{P1p} gives

\begin{Proposition}\label{sto}
 The tritronqu\'ee $h_t$ defined by \eqref{sectorinx} satisfies
 $$h_t(x)=O(x^{-4})\text{ as } x\to +i\infty$$
  (implying $C=0$ in \eqref{eq:eqsm}) and, with  $\mu$ given by \eqref{val_mu},
$$h_t(x)=\mu x^{-\frac12}e^{-x}(1+o(1))\text{ as } x\to -i\infty$$
\end{Proposition}
 The fact that $\mu$ is the same for all tronqu\'ees including the tritronqu\'ee is general - see in \cite{imrn} eq. (1.19) and more generally in \cite{duke} eq. (1.27). It is clear from Proposition \ref{sto} and Theorem \ref{P1p} that it suffices to obtain $\mu$
for $h_t$.
The proof of Proposition\,\ref{sto} is given in Section \ref{find}; the
equation that $\mu$ solves is \eqref{stok2}.

\subsection{Sectors with singularities. Setting and heuristics.}\label{KAM}

\subsubsection{Arrays of poles near regular sectors of tronqu\'ees solutions}\label{firstpoles}
Solutions $h$ satisfying \eqref{eq:eqsm} are analytic for large $x$ in the
right half plane. Beyond the edges of the sector  $-\pi/2\leqslant{\rm{arg}}\,
x\leqslant\pi/2$, $h$ develops arrays of poles (unless $h$ is
tritronqu\'ee). These facts are proved, together with the location of the
first few arrays of  singularities, in \cite{invent} and \cite{cch} and
are overviewed below.

Given $h$ as in Theorem \ref{P1p} there is a unique constant $C_+$ with the
following properties. Denoting  $\xi=\xi(x)=C_+x^{-1/2}e^{-x}$ the leading
behavior of $h$  for large $|x|$ with ${\rm{arg}}\, x$ close to $\pi/2$ is
 \begin{equation}
   \label{eq:eq51}
   h\sim H_0(\xi)+\frac{H_{1}(\xi)}{x}+\frac{H_{2}(\xi)}{x^2}+\cdots\ \ (x\to i\infty \text{ with }|\xi-12|>\epsilon,\ |\xi|<M)
 \end{equation}
(if $\xi$ is small, the terms
may need to be reordered) where
\begin{equation}\label{eq:f11}
 H_0(\xi)= \tfrac{\xi}{(\xi/12-1)^2},\, H_{1}(\xi)=\tfrac{
-\frac{1}{60}\xi^4+3\xi^3+210\xi^2
+216\xi}{(12-\xi)^3},\ldots,\, H_{n}(\xi)=\tfrac{P_n(\xi)}{\xi^n(\xi-12)^{n+2}}
\end{equation}
with $P_n$ polynomials of degree $3n+2$. \footnote{ These can be obtained by substituting the expansion \eqref{eq:eq51} in \eqref{eq:eqp}, treating $x$ and $\xi$ as independent variables, and solving order by order in $1/x$.}

The first array of poles beyond $i\RR^+$ is located at points $x=p_n$ near the
solutions $\tilde{p}_n$ of the equation $\xi(x)=12$, namely
\begin{equation}
  \label{eq:pospoles}
  p_n=\tilde{p}_n+o(1)=2n\pi i -\frac{1}{2}\ln(2n\pi i)+\ln C_+-\ln\, 12+o(1),\ \ \ \ (n\to\infty)
\end{equation}
Rotating $x$ further into the
second quadrant, $h$ develops
successive arrays of poles   separated by distances
$O(\ln x)$ of each other as long as $\arg(x)=\pi/2+o(1)$ \cite{invent}.
\begin{Note}
  {\em  The array of poles developed near the other edge of the sector of
    analyticity, for $\arg(x)=-\pi/2+o(1)$, is obtained by the conjugation
    symmetry in \S\ref{sym}: in  \eqref{eq:eq51} and \eqref{eq:pospoles}
    $i$ is
    replaced by $-i$ and $C_+$ by a different constant, still unique,
    $C_-$.  The tritronqu\'ee $h_t$ has the sector of analyticity as in
    \eqref{sectorinx}; $h_t$ has an array of poles for
    $\arg(x)=-\pi/2+o(1)$. For more details see \cite{cch}. }
\end{Note}

\subsubsection{Sectors with poles. Setting and heuristics} As
mentioned, the general solution of P$_1$ has poles in any sector in $\CC$, and
any solution has at least a sector of width $2\pi/5$ with singularities.  In particular, any truncated solution $h$ as in Theorem\,\ref{P1p} has poles outside the sector $-\pi/2\leq{\rm arg }\,x\leq \pi/2$ and $h_t$ has poles outside the sector $-\pi/2\leq {\rm arg }\,x\leq 3\pi/2$, in particular $h_t$ has poles for  $x$ in the sector
 \begin{equation}\label{def-Sigma}
  \Sigma=\{x\, |\, -\pi<{\rm{arg}}\, x<-\pi/2\}
  \end{equation}

The heuristics of the approach in the present paper are given in detail in \cite{cch}, see esp. \S 3.2. As
it is often the case, rigorous arguments are more involved and sometimes
depart from the heuristic ideas. These are given in \S\ref{xjc}. With $u:=h$,
\vspace{-0.25cm}\begin{equation}
  \label{def_s}
  s={h'}^2-h^2-h^3/3
\end{equation}\vspace{-0.25cm}
and
\vspace{-0.25cm}\begin{equation}\label{def_R}
R(u,s)=\sqrt{u^3/3+u^2+s}
\end{equation}\vspace{-0.25cm}
equation (\ref{eq:eqp}) can be rewritten as a system
\begin{align}
  \label{eq:eqds2n}
 & s(u)=s_n-2\int_{u_n}^u \(\frac{R(v,s(v))}{x(v)}-\frac{392}{625}\frac{1}{x(v)^4}\)dv \\
 & x(u)=x_n+\int_{u_n}^u \frac{1}{R(v,s(v))}dv \label{eq:eqdx2n}
\end{align}
where the integrals are along a closed curve $\mathcal{C}$ (see Note \ref{chooseC}), we write $u_n$ to denote that $u$ has traveled $n$
times along $\mathcal{C}$, and $s_n = s(u_n),~ x_n = x(u_n)$. See also \S\ref{return_map}.

The functions
\begin{equation}
  \label{eq:eqJL}
    \JN (s):=\oint _\mathcal{C} R(v,s)\,dv;\ \ \LN(s):=\oint_\mathcal{C} \,\frac{dv}{R(v,s)}
\end{equation}
satisfy the equations\vspace{-0.25cm}
\begin{equation}
  \label{eq:difeq00}
 \JN\,''+\frac{1}{4}\rho(s)\JN=0;\ \ \text{where}\ \  \rho(s)=\frac {5}{3s \left( 3\,s+4 \right) }
\end{equation}\vspace{-0.25cm}
and\vspace{-0.25cm}
\begin{equation}
  \label{eq:difeq01}
  \LN\, ''-\frac{\rho'(s)}{\rho(s)}\LN\, '+\frac{1}{4}\rho(s)\LN=0
\end{equation}\vspace{-0.25cm}

A natural procedure relying on the Poincar\'e return map leads to two candidates for asymptotically conserved quantities, see \cite{cch},
\begin{equation}
  \label{eq:eqapprx}
  \mathcal{Q}(x,s):= x\JN(s)= x_0\JN(s_0)\, \left(1+o(1)\right)
\end{equation}
and
\begin{equation}\label{eq:2ndc}
 \mathcal{K}(s)+\frac{2n}{\kappa_0\, x_0\JN(s_0)}=\mathcal{K}(s_0)+o(1)
\end{equation}
for $n=O(x_0)$, where
\begin{equation}\label{defK}
\mathcal{K}(s):=\kappa_0\int_{0}^s\frac{ds}{\JN(s)^2}
=\frac{\hat{J}(s)}{\JN(s)}
\end{equation} $\hat{J}$ is an independent solution of \eqref{eq:difeq00} with $\hat{J}(0)=0$, and $\kappa_0$ is
the Wronskian of $\JN$ and $\hat{J}$.

\begin{Note}\label{nsing}
  {\em   Near the edge
$\arg x =-\pi/2-o(1)$ of $\Sigma$ (or near the other edge, for $\arg x =-\pi+o(1)$) special care is
needed for the truncated solutions since $s=o(1/x)$ (respectively, $s=-\frac43+o(1/x)$), so $s$ is near
singularities.}
\end{Note}

\begin{Note}\label{NewNote}
 {\em One gets higher orders in the asymptotic expansions of $(x_n,s_n)$ by formal
Picard iterations, using \eqref{eq:eqapprx} and \eqref{eq:2ndc} in
\eqref{eq:eqdx2n} and \eqref{eq:eqds2n}; these lead, after inversion of
\eqref{eq:eqdx2n} and \eqref{eq:eqds2n}, to an asymptotic expansion of $h(x)$
and $h'(x)$. This is quite straightforward and fairly short, but a formal calculation will introduce uncontrolled errors, and a good part of the technical
sections of the paper deals with rigorizing the analysis.}
\end{Note}

\subsection{Calculating $\mu$} Let $h=h_t$.
After having obtained and proved the
asymptotic expansions of $h$ and $h'$ as in Note \ref{NewNote}, we match them to
expansions  of type
\eqref{eq:eq51}, as explained in \S\ref{firstpoles}, of $h$ and $h'$ when $\arg x$ decreases below $-\pi$.  The matching, relying on the  single-valuedness of $y(z)$ which entails a consistency condition of the asymptotics
translates into an equation for $\mu$, explained in \S\ref{find}  (in particular, see Note \ref{method} and \eqref{stok2}).  Since $h_t$ is a tritronqu\'ee, a unique solution regular for $\arg
x\in (\pi/2,-\pi/2)\cup (-\pi,-3\pi/2)$, not surprisingly, the equation for
$\mu$ has a unique solution ($\mu=\sqrt{\frac{6}{5 \pi }}i$, see \S\ref{find} for details).

\begin{Note}
The validity of our method does not rely on integrability. For instance, if we drop the term $-\frac{392}{625 x^4}$ in \eqref{eq:eqp}, the
equation becomes Painlev\'e nonintegrable, but our asymptotic expansion
of $(s_n,x_n)$ does not change to the order used. However, for such nonintegrable equations, $\mu$ cannot be identified as the Stokes
multiplier, since analytic continuations would involve different Riemann sheets and no obvious matching would be possible.
%Application of our methods to more general equations including P$_2$ are discussed in \S\ref{secp2}.
\end{Note}
\section{Main results about the singular sectors}\label{singreg}
\subsection{Asymptotic expansions of solutions}
We start with some fixed $u_0,s_0,x_0$  and
as $u$ travels along $\mathcal{C}$ back to $u_0$ we obtain $s_1,x_1$, and
repeating this procedure we get $s_2,x_2,$ etc. (see also \eqref{eq:eqds2n} and \eqref{eq:eqdx2n} and comments following these equations) In this way
\eqref{eq:eqds2n} and \eqref{eq:eqdx2n} provide a recurrence relation
describing the evolution of $s$ and $x$ as $u$ goes along $\mathcal{C}$.

Theorem\,\ref{recu1} provides asymptotic expansions of the asymptotically conserved quantities \eqref{eq:eqapprx}, \eqref{eq:2ndc} of the system \eqref{eq:eqds2n}, \eqref{eq:eqdx2n}, up to $O(x^{-5/4})$ or better.

It turns out (Lemma \ref{Consistency1})  that for $h=h_t$ the $s_n$ are in the upper half plane with modulus proved to be
less than $5$ (it is numerically $\le 2$), $s\in \mathbb{D}_5^+$ cf. \eqref{eq:defdd}.
 Since $J$ has singularities
(square root branch points) at $s=0$ and at $s=-4/3$ only, both $J$ and $L$ are single-valued for
$s\in \mathbb{D}_5^+$.

\subsubsection{Choosing initial data $s_0,x_0,u_0$}\label{Assumption1}

 {\em General strategy.} We obtain asymptotic expansions for large $x$, therefore $x_0$ will be chosen large enough. Also, $x_0$ will be chosen near $i\RR_-$, the lower edge of the sector of analyticity of the tritronqu\'ee solution $h_t$.
Choosing $s_0$ near $0$ (a singularity), and iterating the Poincar\'e map we will obtain $x_0,x_1,\ldots,x_n,\ldots,x_{N_m}$ which go through the sector with singularities up to the other edge, as it will be proved that arg$\,x_{N_m}$ is close to $-\pi$. It will turn out that $s_{N_m}$ is close to the other singularity, $-4/3$.

The iteration can also be done for other values for $s_0$, not necessarily close to singularities, in which case the estimates are simpler, and they can be used to obtain asymptotic conserved quantities in sectors with poles for any solution of P$_I$. In this paper however, we are interested in the connection problem, and then we do need $s_0$ close to singularities; see also Note \ref{Nnond}. More precisely:

{\bf{Assumption.}}{\em{ We choose $\mathfrak{m}>0$ a large enough number (a concrete estimate can be obtained by tracing the calculations involving it) and $x_0$ so that
\begin{equation}\label{ASSS1}
|x_0|>\mathfrak{m},\  \Im x_0<0 ,\  \ |\Re x_0|<\ln |x_0| ,\  \\ \ \Im s_0>0 ,\ \\ \ 1<|s_0x_0|<10,\ \ u_0=-4
\end{equation}}}
The choice of $u_0=-4$
makes some calculations simpler, cf. Note \ref{u04}, though we will use this precise value only later.

Note that \eqref{ASSS1} implies that $x_0=-i|x_0|e^{i\theta_x}$ where $\theta_x$ is $o(1)$, so indeed, $x_0$ is close to $i\RR_-$ and that $s_0=O(x_0^{-1})$, close to $0$ indeed.

\subsubsection{The solutions $J$, $\hat{J}$ used and other notations}\label{SolJhJ}

In the rest of the paper $J$ is the unique solution of \eqref{eq:difeq00} satisfying
\begin{equation}
  \label{eq:FJB}
  J(s)=2A \left[ 1-\left( {\frac {5}{96}}+{\frac {5}{48}}\,\ln   s
 \right) s \right] +2sB+O(s^2)
\end{equation}
with
\begin{equation}
  \label{eq:solAB}
 A=-\frac{12}{5};\ B=-\frac{3}{8}-\frac{1}{2}\ln(24)+\frac{\pi i}{4}
\end{equation}
and $\hat{J}(s)$ is the unique solutions of \eqref{eq:difeq00} with
\begin{equation}
  \label{eq:defJh}
 \hat{J}(0)=0, \ \hat{J}'(0)=\pi i
\end{equation}

Denote \begin{equation}
  \label{eq:eqQ}
  Q(u,s)=\frac{2}{3}\frac {6\,s+6\,su+18\,u+3\,u^2-4\,u^3-u^4}{s \left(
3\,s+4 \right) \sqrt {9\,s+9\,u^2+3\,u^3}}=\rho(s)\,\frac{P(u,s)}{R(u,s)}
\end{equation}
where $R$ is given by \eqref{def_R}, $\rho$ by \eqref{eq:difeq00}
and the polynomial $P$ equals
\begin{equation}
  \label{eq:eqP}
  P(u,s)=\frac{2}{15}\Big[u(2-u)(3+u)^2+6s(1+u)\Big]
\end{equation}
Denote
\begin{multline}\label{defs}
  N_0=\lfloor|x_0|^{\frac34}\rfloor,\  c_0=\frac{e^{-\frac{\pi
       i}{6}}}{\sqrt{3}},\  \zeta=\frac{5s_0x_0}{48},\ \tilde{\zeta}=\frac{5is_{_{N_m}}x_{_{N_m}}}{48},\ {n}'={n}+{\zeta},\
\tilde{n}'={n}+\tilde{\zeta}
\end{multline}
and define
 \begin{align}
  \label{eq:defbn}
 & B_{n-1}=\tfrac{5}{24}\left(n'\ln\tfrac{n'}{e}-\ln\Gamma(c_0+n')+\ln\Gamma(\zeta+c_0)-\zeta\ln\tfrac{\zeta}{e}\right)
  \nonumber\\
&\tilde{B}_{n-1}=-\tfrac{5}{24}\left(\tilde{n}'\ln\tfrac{\tilde{n}'}{e}-\ln\Gamma(\tfrac12+\tilde{n}')+\ln\Gamma(\tilde{\zeta}+\tfrac12)-\tilde{\zeta}\ln\tfrac{\tilde{\zeta}}{e}\right)
\end{align}
As $n\to\infty$ we have
\begin{align}
  \label{eq:defgab}
& B_{n-1}=\tfrac{24}{5}\left[\tfrac{i\sqrt{3}\ln n}{6}+g_a+O(n^{-1})\right],\ {\rm{where}}\  g_a:=-\zeta\ln\tfrac{\zeta}{e}+\ln\Gamma(\zeta+c_0)-\tfrac{1}{2}\ln(2\pi)\\\nonumber
 & \tilde{B}_{n-1}=\tfrac{24g_b}{5}+o(1),\ {\rm{where}}\
 g_b:=\tilde{\zeta}\ln\tfrac{\tilde{\zeta}}{e}-\ln\Gamma(\tilde{\zeta}+\tfrac12)+\tfrac12 \ln(2\pi)
\end{align}

\begin{Theorem}\label{recu1}
Under the assumption \eqref{ASSS1}, there exists a curve  $\mathcal{C}$ such that the following hold as $|x_0|\to \infty$.

\z (i) The system \eqref {eq:eqds2n},\eqref{eq:eqdx2n} has a unique
solution $(s,x)$ for $u$ traveling $n$ times on $\mathcal{C}$ with $0\Le n< N_m$, where $N_m\in\NN$ is the unique number
such that $0<\Im s_{N_m}<11/|x_0|$ and $|\Re s_{N_m}+4/3|<2|x_0|^{-1/2}$.

\z (ii) We have (see \eqref{defK} for notations)
 \begin{equation}
  \label{eq:eq44d1}
\mathcal{K}_n= \mathcal{K}(s_0)+\frac{48\pi in}{J_0x_0}+
\frac{2\pi i \phi_{n}}{x_0J_0}+O(x_0^{-5/4}\ln x_0)
\end{equation} where
$$\phi_{n}= \frac{n}{x_0}\(g_a+\frac{4\sqrt{3} i}{5}\ln\frac{5s_0x_0}{48}\)+\frac{1}{4\pi
    i}\int_{s_{0}}^{s_n}Q(u_0,s)\(-\frac{2\pi in}{x_0}J(s)-\hat{J}(s)\)ds $$

\z (iii) For $n\in(N_0,N_m-N_0]$,  $\mathcal{G}_n:=\mathcal{Q}_n/x_0$ (for notations see \eqref{eq:eqapprx}, \eqref{defs}) we have
\begin{equation}
  \label{eq:q3}
  \mathcal{G}_n= \mathcal{G}_{N_0}
-\tfrac{1}{2}x_0^{-1}\int_{s_{N_0}}^{s_{n}}Q(u_0,s)J(s)ds+O(x_0^{-3/2})
\end{equation}
while for $n\in [0,N_0]$  (resp. $n\in(N_0,N_m-N_0]$) where $s_n$ is small, $\mathcal{G}_n$ is given by
\begin{equation}\label{q3}
\mathcal{G}_n=\mathcal{G}_0+x_0^{-1}B_{n-1}+O(x_0^{-\frac54}\ln x_0),\ \text{resp.}\
\mathcal{G}_{N_m}-x_0^{-1}\tilde{B}_{N_m-n+1}+O(x_0^{-\frac54}\ln x_0)
\end{equation}
\end{Theorem}
Part (i) follows from Proposition \ref{Wholeinterval}, (ii) is proved in
\S\ref{partiii} and  (iii) is shown in
\S \ref{partii}, with \eqref{eq:defbn},\eqref{eq:defgab}  shown in Lemma \ref{valgab}. Higher order
corrections can be obtained in the usual asymptotic way, iteratively
order-by-order.
\begin{Note}\label{Nnond}{\rm
 In fact, \eqref{q3} applies to more general conditions $s_{N0}\in\HH$,
if  $s_n\in\HH$ for $n=0,...,N_m-N_0$.}
\end{Note}

Based on Theorem
\ref{recu1} (ii) and (iii) more orders can be obtained for $s_n$ and $x_n$:
\begin{Proposition}\label{c123}
(i) For $N_{m}-N_0\Le n \Le N_{m}$ we have
\begin{equation}\label{snn}
s_n=-\frac43- s_0
+\frac{24i}{5\pi}+\frac{1152n}{25J_{0}x_{0}}-\frac{2\phi_n}{x_0^2}+O\(x_0^{-2}\ln
^2x_0\)+O\((s_{n}\h,^-)^{3/2}\)
\end{equation}
and thus $N_{m}=\frac{|x_0|}{2\pi}+O(\ln x_0)$.

(ii) We have\vspace{-0.25cm}
\begin{equation}\label{sn8}
\Re s_{N_{m}}=-\frac43-\Re s_0+\frac{\Re (x_0J_0)}{\pi|x_0|} +O(x_0^{-1})
\end{equation}\vspace{-0.25cm}
(iii) Also\vspace{-0.25cm}
\begin{equation}\label{xnm1}
x_{N_{m}}=\frac{x_0J_0}{J_{N_{m}}}+\frac{\sqrt{3}}{6}\ln x_0+O(1)
\end{equation}\vspace{-0.25cm}
In particular
$$x_{N_{m}}=-i x_0+O(\ln x_0)\ \ \  \ \   \rm{and}    \ \ \ \ \
\Im x_{N_{m}}=\Im \frac{x_0J_0}{J_{N_{m}}}+O(1)$$\vspace{-0.25cm}
\end{Proposition}

The proof is given in \S\ref{Pfc123}.

\section{Proofs. I. General properties of the functions used in the proofs}\label{GenProp}

\subsection{The zeroes of $R(u,s)$}\label{roots_prop}
We denote
\begin{equation}
  \label{eq:defdd}
 \mathbb{H}=\{z\in\CC\, \big|\, \Im z>0\},\ \  \mathbb{D}^+_\rho=\{s\in\mathbb{H}\, \big|\,|s|<\rho\}\, ,\ \ \  \mathbb{D}^-_\rho=\{s\in-\mathbb{H}\, \big|\,|s|<\rho\}
\end{equation}
Denote by $P_s(u)$ the following polynomial in $u$, with parameter $s\in \mathbb{H}$
 \begin{equation}\label{defPsofu}
 P_s(u)=u^3/3+u^2+s
\end{equation}
 and note the symmetry
$P_s(u)=-P_{-s-4/3} (-u-2)$,
or
\begin{equation}\label{sym_T}
P_s(u)=-P_{-\sm} (-\um)\ \ \ {\rm{where\ \ }}\sm=s+\frac43,\ \  \ \um=u+2
\end{equation}
which entails that results for $s$ close to $0$ can be translated into results for $s$ close to $-\frac43$.

The only values of $s$ for which two roots of
$P_s(u)$ coalesce are $s=0$ and $s=-4/3$. Therefore the roots $r_{1,2,3}(s)$ of $P_s(u)$ are distinct and analytic for $s\in  \mathbb{H}$ (see, e.g. \cite{Ahlfors}). Lemma\,\ref{lem1} gives bounds for these roots and for distances between them, see Fig.\,\ref{fig:ro}.

\begin{Lemma}\label{lem1}

(0) For $s\in\mathbb{H}$ we have
\begin{equation}
  \label{eq:eqr1}
  \begin{array}{ll}
   \ro_1\in -\mathbb{H}, & \Re \ro_1<-2 \\
   {\ro}_2\in \mathbb{H}, & \arg {\ro}_2>\arctan(3/2),\ \
  \arg({\ro}_2+2)<\pi-\arctan(3/2)\\
{\ro}_3\in -\mathbb{H}, & \Re {\ro}_3  >0
\end{array}
\end{equation}

(i)  For  $|s|<\mathbb{D}^+_{\sqrt{\tfrac23}}$ (for notation see \eqref{eq:defdd})
with
the choice $\sqrt{s}>0$ if $s>0$ (and a choice of labeling of the roots) we have
\begin{equation}
  \label{eq:eqr}
|\ro_1+3+s/3|<|s^2| ,\ \ \ \    |r_{2}-i\sqrt{s}|<|s|,\ \ \ \ |r_{3}+i\sqrt{s}|<|s|
\end{equation}

(ii) For  $\sm\in \mathbb{D}^+_{\sqrt{2/3}}$ we have
\begin{equation}
  \label{eq:eqrt0}
 |2+\ro_1+\sqrt{\sm}|<|\sm|,\ \ \ \      |2+{\ro}_2-\sqrt{\sm}|<|\sm|,\ \ \ \
 |{\ro}_3-1+\sm/3|<|\sm^{\!\!\!2}|
  \end{equation}

(iii) Let $r_{j}(s;t)$ be the roots of $\ {tu^3}/{3}+u^2+s$, labeled
with the convention $r_j(s;1)=r_j(s)$.

If $|s|\in\mathbb{D}^+_{1/10}$ then $r_{{\tcT}}(s;t)$ are real analytic in $t\in
(0,1)$.

If $|\sm|\in\mathbb{D}^+_{1/10}$ then $r_{1,2}(s;t)$ are real
analytic in $t\in (0,1)$.

(iv) If $s\in \newD6p $ then $|r_{1,2,3}(s)|<399/100$.
\end{Lemma}
% let $P_{t,u,s}=tu^3^3 u^2+4/3+s$.  Suppose $|s|<1/10$ and The
%roots
%of $P_{t,u,s}$ are real analytic in $t\in (0,3)$.

\begin{figure}
  \centering
\includegraphics[scale=0.8]{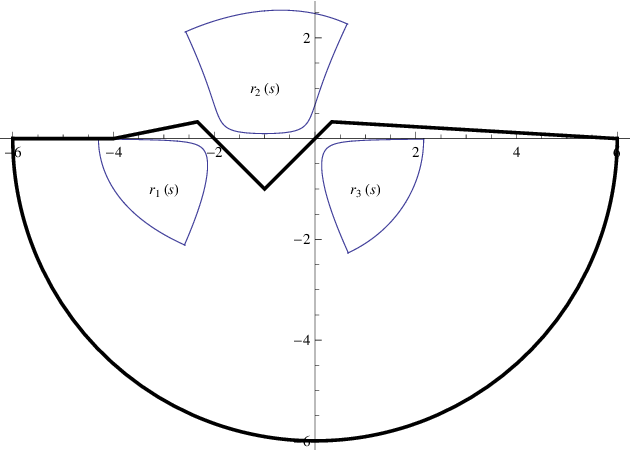}
  \caption{Regions of the roots of $u^3/3+u^2+s$ and the contour $\mathcal{C}$}
  \label{fig:ro}
\end{figure}

The proof, elementary, but rather laborious, is given in the Appendix, \S\ref{appendix}.

\begin{Note}[{\rm{Choosing the closed path.}}]\label{chooseC}
Let  $\mathcal{C}$ consist of the polygonal line
connecting
$6,\, \frac{1+i}{3},$ $ -(1+i),-2+\frac{-1+i}{3},\, -3+\frac{i}{3},\, u_0,\,-6$ and a semicircle of radius 6 centered at the origin in the lower half plane, see
Figure \ref{fig:ro}. For us $u_0=-4$, cf. \S\ref{Assumption1}.

\end{Note}

We will also need the incomplete integrals $J$ and $L$:
\begin{equation}
  \label{eq:eqJLu}
  J(u,s)=\int_{u_0}^uR(v,s(v))\,dv;\ \ L(u,s)=\int_{u_0}^u\,\frac{dv}{R(v,s(v))}
\end{equation}
where the integration is along $\mathcal{C}$ (specified above),  which will be shown to surround  two of the three roots of $u^3/3+u^2+s_0$. The contour is traveled upon multiple times, and we will use an index $n$
to specify the winding number.
\begin{Corollary}\label{lem2}
 (i)  We have
  ${\ro}_3(s),\,\ro_1(s)\in \DD^-_5$ while ${\ro}_2(s)\in {\DD^+_5}$ for all $s\in\newD6p$.

(ii) Consider the polygon $\hat{\mathcal{C}}_0$ with vertices
$-1,-1-6i,6-6i,6+6i,-6+6i,-6,-1$ oriented anticlockwise. Then ${\ro}_2(s),\,{\ro}_3(s)\in\text{int}(\hat{\mathcal{C}}_0)$ while
  $\ro_1(s)\in\text{ext}(\hat{\mathcal{C}}_0)$ if  $s\in\newD6p$.

(iii)  For all  $s\in
\newD6p$ the path $\mathcal{C}$ defined in Note\,\ref{chooseC} encloses $r_1(s),\, r_3(s)$, leaving $r_2(s)$ outside. Moreover
\begin{equation}\label{alp}
\alpha:=\sup_{u\in\mathcal{C},s\in\newD6p} |J(u,s)|<\infty
\end{equation}
where ${J}$ is as defined in \eqref{eq:eqJLu}.
\end{Corollary}

\begin{proof}
 By Lemma \ref{lem1} we have $|r_j(s)|<5$, $j=1,2,3$,   $\Im r_{{\Tco}}(s)<0$ and
 $\Im {\ro}_2(s)>0$ implying (i). Continuity of $J$ is manifest, and  $J(0)$ is an
 elementary integral.

(ii) By Lemma \ref{lem1} we have ${\ro}_1(s)\in\{z:\Im z<0,\Re z<-2\}$ implying
${\ro}_1(s)\in\text{ext}(\hat{\mathcal{C}}_0)$, ${\ro}_2(s)\in\{z:\Im z>0,|z|<5\}$,
and ${\ro}_3(s)\in\{z:\Im z<0,\Re z>0,|z|<5\}$, which implies
${\ro}_2(s)$ and ${\ro}_3(s)$ are in $\text{int}(\hat{\mathcal{C}}_0)$. Continuity
of $\hat{J}$ at zero is manifest, and it implies $\hat{J}(0)=0$; this together
with the fact that $\hat{J}(s)$
satisfies \eqref{eq:difeq00} implies,   by Frobenius theory, that  it is
analytic at zero (see also \S\ref{j043} below);  the value of $\hat{J}'(0)$
is simply obtained by the residue theorem.

(iii) This is obvious by Lemma \ref{lem1} and continuity of $\tilde{J}$.
\end{proof}

\subsection{Link between $J$ and $L$}\label{LinkJL}
We use the notations \eqref{def_R},\,\eqref{eq:eqJL} where $\mathcal{C}$ can be any closed curve  (piece-wise smooth),\eqref{eq:eqQ},\,\eqref{eq:eqP}.

\begin{Proposition}\label{Ident1}
 We have
  \begin{equation}
    \label{eq:iden1}
    \frac{\partial Q}{\partial u}=   \frac{1}{R(u,s)^3}-\rho(s)R(u,s)
       \end{equation}
       %\ \ \ \frac{\partial Q}{\partial u}=   \frac{1}{(u^3/3+u^2+s)^{3/2}}-\rho(s)\sqrt{u^3/3+u^2+s}

In particular,
 using \eqref{eq:eqJL} and \eqref{eq:iden1} we have
\begin{equation}
  \label{eq:id3}
  \frac{dJ(s)}{ds}=\frac{1}{2}L(s);\ \ \ \ \ \ \ \frac{dL(s)}{ds}=-\oint_\mathcal{C}\, \frac{du}{2R^3(v,s)}=-\frac{1}{2}\rho(s)J(s)
\end{equation}

\end{Proposition}

The proof of Proposition\,\ref{Ident1} is by direct verification. \Bx

\subsection{Integral representations of $J$, $\hat{J}$ , $L$ and $\hat L$}\label{j043}
 We defined $J$ and $\hat{J}$ as solutions of
 \eqref{eq:difeq00} satisfying the initial conditions \eqref{eq:FJB}-\eqref{eq:defJh}; we now derive some integral representations useful in the sequel.

Denote by $\kappa$ is the elliptic modulus
\vspace{-0.25cm}\begin{equation}\label{kap}
\kappa=\frac{{\ro}_1-{\ro}_3}{{\ro}_2-{\ro}_3}.
\end{equation}\vspace{-0.25cm}

\begin{Lemma}
  {\rm The points ${\ro}_1,{\ro}_2,{\ro}_3$ are collinear only if
$s\in (-4/3,0)$. The roots $r_j$ are analytic in $s$ except
for $s\in \{-4/3,0\}$. Furthermore, the triangle $\Delta [{\ro}_1,{\ro}_2,{\ro}_3]$ preserves its orientation when
$s$ traverses any curve $\gamma\subset\CC$ which does not cross the real line.

As a consequence,  for $s\ne -4/3,0$ the roots satisfy}
\begin{equation}\label{r123}
r_3{\mbox{ does not belong to the segment  }}[{\ro}_1,{\ro}_2]\subset\CC
\end{equation}

\end{Lemma}

\begin{proof} By Vieta's formulas,  ${\ro}_2+{\ro}_3+{\ro}_1=-3,\ {\ro}_2{\ro}_3+{\ro}_2{\ro}_1+{\ro}_3{\ro}_1=0,\ {\ro}_2{\ro}_3{\ro}_1=-3s$.
%\begin{equation}\label{A1}
 %\end{equation}
and a straightforward calculation gives
$${\ro}_2=-1+\frac{\kappa-2}{\sqrt{\kappa^2-\kappa+1}},\ \ {\ro}_3=-1+\frac{1+\kappa}{\sqrt{\kappa^2-\kappa+1}},\ \ {\ro}_1=-1+\frac{1-2\kappa}{\sqrt{\kappa^2-\kappa+1}}$$

If $r_{1,2,3}$ are colinear then $\kappa\in\RR$ which in turn implies  $r_{1,2,3}\in\RR$, hence $s\in (-4/3,0)$, in which case $r_1<r_2<r_3$.

  Analyticity is standard: the roots satisfy $F(r,s)=0$ $, r=r_j$,
  which, by the implicit function theorem defines analytic functions
  $r_j(s)$ in a neighborhood of any point where $F_{r}\ne 0$. But
  $F_r(r_j)=0$ clearly means that the polynomial has a double root. If we take a
curve $\gamma$ not intersecting $S_0$, then ${\ro}_1$ and ${\ro}_2$ are always distinct,
and we can orient the line through ${\ro}_1$ and ${\ro}_2$ by choosing the direction
from ${\ro}_1$ to ${\ro}_2$ as being positive. If ${\ro}_3$ is, for some $s$, to the left of the line (in the usual meaning) it stays to the left by continuity,
since the distance between $s_0$ and the line is never zero.
\end{proof}

\begin{Proposition}\label{js0}
(i)   We have $L=2L_{3,1}$, $J=2J_{3,1}$, $\hat{J}=2J_{{3,2}}$ and
  $\hat{L}=2L_{3,2}$,
where
\begin{equation}
  \label{js01}
  L_{i;j}=\int_{r_i}^{r_j}\frac{1}{R(u,s)}du\,\ \ \ \ \ \ \ \ \ \ \ \ \  J_{i;j}=\int_{r_i}^{r_j}R(u,s)du
\end{equation}
with the branch of $R(u,s)=\sqrt{P_s(u)}=3^{-1/2}\sqrt{(u-{\ro}_1)(u-{\ro}_2)(u-{\ro}_3)}$ defined so that {\rm{arg}}$\,P_s(u)=0$ for $u\to+\infty$, for $s\in\RR$ {\rm{arg}}$\,P_s(u)=0$ for $u>r_3$,
 {\rm{arg}}$\,P_s(u)=\pi$ for $r_2<u<r_3$,  {\rm{arg}}$\,P_s(u)=2\pi$ for $r_1<u<r_2$,  and {\rm{arg}}$\,P_s(u)=3\pi$ for $u<r_1$. For $s$ in the upper half plane the branch of $R(u,s)$ is defined by  analytic continuation.

(ii)
We have, with the notation \eqref{kap},
\begin{equation}\vspace{-0.25cm}
  \label{eq:t23}
J_{3,1}=  -3^{-1/2}({\ro}_1-{\ro}_3)^{5/2}\int_0^1\sqrt{t(1-t)( \kappa^{-1}-t)}dt
\end{equation}\vspace{-0.25cm}
\begin{equation}\vspace{-0.25cm}
  \label{eq:t21}
J_{3;2}=-3^{-1/2}({\ro}_2-{\ro}_3)^{5/2}\int_0^1\sqrt{t(1-t)(\kappa-t)}dt
\end{equation}\vspace{-0.25cm}
(iii) As $\sm=s+\tfrac43\to 0$ we have
\begin{equation}\label{js43}
J(s)=-\frac{24i}{5}-\(\tfrac12 \pi-i\ln 24 -\tfrac12 i\)\sm -\tfrac12
i \sm \ln(-\sm)+O(\sm^{\!\!\!\!2}\ln \sm)
\end{equation}

\begin{equation}\label{ls43}
L(s)=-i\ln \sm-\pi+2i\ln24 +o(1)
\end{equation}
where $\sm=s+4/3$ is as defined in Lemma \ref{lem1}.

(iv) As $s\to 0$ and $\sm\to 0$ resp. we have
\begin{equation}
  \label{eq:t21a}
\hat{J}(s)= \pi i s+O(s^2\ln s);\ \hat{J}(s) =J(s)-2J_{{\tso}}=J(s)+\pi \sm+O(\sm^{\!\!\!2}\ln \sm)\end{equation}
In particular we have \eqref{eq:FJB} and \eqref{eq:defJh}.
\end{Proposition}
For the proof, we note that $J_{i;j}$ are solutions of  \eqref{eq:difeq00}. To
identify them, we simply have to determine their behavior at $s=0$. We note that $\sqrt{P_0(u)}$ has a square root singularity
at $u=-3$ thus $\int_{-3-3s}^{-1}1/R(u,s)$ =$\int_{-3}^{-1}1/R(u,0)+o(1)$ as $s\to 0$.
Thus
\begin{equation}
  \label{eq:eqJ1}
  L_{3,1}=\int_{{\ro}_3}^{-3}\frac{du}{R(u,s)} +o(1)\ \ \ \ (s\to 0)
\end{equation}
We re-express $R(u,s)$ as its approximation where we discard $u^3$
plus the corresponding difference:
\begin{equation}
  \label{eq:decom}
 \frac{1}{\sqrt{u^3/3+u^2+s}}= \frac{1}{\sqrt{u^2+s}}+Q_1(u,s)
\end{equation}
where
\vspace{-0.25cm}\begin{equation}
  \label{eq:eqQ_1}
  Q_1(u,s):=-\frac{3u^3}{\sqrt{3u^3+9u^2+9s}\, \sqrt{u^2+s}\, \left(3\sqrt{u^2+s}+\sqrt{3u^3+9u^2+s}\right)}
\end{equation}\vspace{-0.25cm}
We note that $Q_1$ is continuous at $s=0$,
and, for $u<0$,
\begin{equation}
  \label{eq:q10}
  Q_1(u,0)={\frac {\sqrt {3}}{\sqrt {u+3} \left( 3+\sqrt {3}\sqrt {u+3} \right)
}}
\end{equation}
 and thus
\vspace{-0.25cm}\begin{multline} \label{eq:eq4}
   L_{3,1}=\int_{{\ro}_3}^{-3}\frac{1}{\sqrt{u^2+s}}\, du+\int_{r_3}^{-3}\,  Q_1(u,0) \, du +o(1)\\
   =\int_{{\ro}_3}^{-3}\frac{1}{\sqrt{u^2+s}}\, du+\int_{0}^{-3}\frac {\sqrt {3}}{\sqrt {u+3} \left( 3+\sqrt {3}\sqrt {u+3} \right)}\, du +o(1)
\end{multline}\vspace{-0.25cm}
With the change of variable $u=-i\sqrt{s}\,\,\sinh v$ we get
\begin{equation}
  \label{eq:eqin1}
  \int_{-i\sqrt{s}}^{-3}\frac{1}{\sqrt{u^2+s}}du=-\int_{\frac{i\pi}{2}}^{\sinh^{-1}(3/\sqrt{s})}dv=\frac{1}{2}i\pi-\sinh^{-1}(3/\sqrt{s})=\frac{\pi i}{2}-\ln 6+\frac{1}{2}\ln s
\end{equation}
The second integral in \eqref{eq:eq4} is $-2\ln 2$ and thus
\begin{equation}
  \label{eq:finL23}
  L_{3,1}=\frac{1}{2}\ln s-\ln(24)+\frac{\pi i}{2}+o(1)
\end{equation}
Applying Frobenius theory to \eqref{eq:difeq00} we get
\begin{equation}
  \label{eq:FJ}
  J_{3,1}=A \left[ 1-\left( {\frac {5}{96}}+{\frac {5}{48}}\,\ln   s
 \right) s+\cdots \right] +B \left( s-{\frac {5}{96}}\,{s}^{2}+\cdots
 \right)
\end{equation}
We have $L_{3,1}=2J_{3,1}'$, and thus
\begin{equation}
  \label{eq:cfJ}
  L_{3,1}=2\,A \left( -{\frac {5}{32}}-{\frac {5}{48}}\,\ln  \left( s \right)
 \right) +2\,B+o(1)
\end{equation}
Comparing with \eqref{eq:finL23} we get (\ref{eq:solAB})
and thus  $J_{3,1}$ has the asymptotic expansion \eqref{eq:FJ} with
$A,B$ given by \eqref{eq:finL23}. It follows that $J(s)=2J_{3,1}$. Similarly
one can show that $J_{3,2}(0)=0$ and $L_{3,2}(0)=\pi i$ implying $\hat{J}=J_{{\tsT}}$.

(ii) The change of variable $u={\ro}_3+t({\ro}_1-{\ro}_3)$ transforms $J_{3,1}$ in \eqref{js01}
into \eqref{eq:t23}; the other $J_{i;j}$ are dealt with similarly.

(iii), (iv) The proof is similar to that of (i).

\begin{Lemma}\label{eps1}
There is some $0<\eta_1<1/100$ such that $|s|<2\eta_1$ implies\newline
$\left|J(s)+\frac{24}{5}+\(\ln 24+\tfrac12 -\tfrac12 \pi i \)s-\tfrac12 s\ln s\right|<|s|^{3/2}$
 and
 %\begin{equation}
  % \label{eq:estJ}
  $ |\hat{J}(s)-\pi i s|<|s^{3/2}|$,
% \end{equation}
\newline
whereas $|\sm|=|s+4/3|<2\eta_1$ implies
$\left|J(s)+\frac{24i}{5}+\(\tfrac12 \pi-i\ln 24 -\tfrac12 i\)\sm +\tfrac12
i \sm \ln(-\sm)\right|<|\sm|^{3/2}, and $
 $|J(s)-\hat{J}(s)+\pi \sm|<|\sm|^{3/2}$
and $\Im L(s)>\max(4|\ln s|/5,2|\Re L(s)|)>4$.
In particular $|J(s)+24/5|<\sqrt{|s|}$ for $|s|<2\eta_1$
and $|J(s)+24i/5|<\sqrt{|\sm|}$ for $|\sm|<2\eta_1$.
\end{Lemma}

\begin{proof}
This follows directly from \eqref{eq:cfJ}, Proposition \ref{js0}, \eqref{js43}, \eqref{ls43} and \eqref{eq:t21}.
\end{proof}

\vspace{-0.25cm}
Finally,
\begin{Lemma}\label{bet}
We have $\beta:=\inf_{s\in\newD6p}|J(s)|>0$ and
$\beta_1:=\inf_{s\in\newD6p,|s+4/3|>\epsilon}|J_{{\tso}}(s)|>0$ for any $\epsilon>0$.
\end{Lemma}

\begin{proof}
By the second line of \eqref{eq:eqr1} we have $|{\ro}_1-{\ro}_3|>2$ for all $s\in\mathbb{H}$. By
\eqref{kap} we have $\inf_{s\in\newD6p}|\kappa|>0$.
Now the integral in \eqref{eq:t23} does not vanish for $s\in\overline{\newD6p}$, since the integrand is in the open fourth quadrant for all $t\in(0,1)$; therefore $J\neq 0$. The conclusion then follows from continuity of $J$ in $s\in\overline{\newD6p}$.

The proof for  $J_{{\tso}}$ is similar except that there is a factor $({\ro}_2-{\ro}_1)^{5/2}$ which can vanish when $s=-4/3$, therefore we need the additional condition $|s+4/3|>\epsilon$ in this case.

\end{proof}

\subsection{Conformal mapping of the upper half plane $\mathbb{H}$ by $\mathcal{K}:=\hat{J}/J$}\label{Schwarzian}
After the substitution $s=-4t/3$ equation
\eqref{eq:difeq00} becomes a standard hypergeometric equation
\begin{equation}
  \label{eq:hypg}
  t(1-t)\frac{d^2\JN}{dt^2}-\frac{5}{36}\JN=0
\end{equation}
   where the associated hypergeometric function is degenerate, with $c=0$ \cite{Abramowitz}. Since the conformal map of
ratios of solutions of \eqref{eq:difeq00} does not appear to follow immediately
from standard references such as  \cite{Nehari} or \cite{Abramowitz}, we provide
for completeness an independent analysis.

\begin{Proposition}\label{jjr}{
(i) $\mathcal{K}(s):={\hat{J}(s)}/{J(s)}$ is a conformal map of the upper half plane into the
interior of $\mathcal C_2$ where   $\mathcal C_2$ consists of a semicircle in the upper half plane
centered at $\frac12$ with radius $\frac12$, an arccircle $C_3$ tangent to the imaginary line at
$0$ passing through $e^{-\pi i/3}$ and the  reflection of $C_3$ about $x=\frac12$. In
particular $C_2\bigcap \mathbb{R}=\{0,1\}$; see Fig. \ref{m}.

(ii) If $|s|>5$ then
$\Im \mathcal{K}(s)<-\frac{2}{5}$.
Furthermore,  with
$\eta_1$ as  in Lemma \ref{eps1}, if  $\eta_2>0$ is small enough, then
\begin{equation}\label{jje}
\sup_{|s|<\eta_1,|\sm|<\eta_1,0<\Im s<\eta_2}\left|\Im \mathcal{K}(s)\right|>\eta_2
>0
\end{equation}}

(iii) We have, for small $s$ and $\sm$ resp.,
\begin{equation}
  \label{eq:cK}
  \left|\mathcal{K}(s)+\frac{5\pi i}{24}s\right|\Le |s^{3/2}|,\ \ \ \ \  \left|\mathcal{K}(s)-1-\frac{5\pi i}{24}\sm\right|\Le |\sm|^{3/2}
\end{equation}

\end{Proposition}
We first prove a result for a M\"obius transformation of $\mathcal{K}$.
\begin{Lemma}\label{ffa}
  Let $$f_a(s)=\frac{_2F_1(-\tfrac16,\tfrac56;\tfrac53,-\tfrac{4}{3s})}{s^{2/3} ~_2F_1(-\tfrac56,\tfrac16;\tfrac13,-\tfrac{4}{3s})}$$
Then $f_a$ maps the upper half plane conformally into the interior of
$\mathcal{C}_1$ where $\mathcal{C}_1$ consists of the segment $I_1=[0,a]$,
where the number $a>0$ is given in \eqref{eq:eqvals} below, followed by an arccircle tangent at $1$ to it and at $e^{4\pi i/3}$ to
$I_2=e^{4\pi i/3}I_1$, and then followed by $I_2$. Furthermore, $|f_a(s)|<\frac{1}{4}$ if $|s|>5$.
\end{Lemma}
\begin{figure}
  \centering
 \includegraphics[scale=0.2]{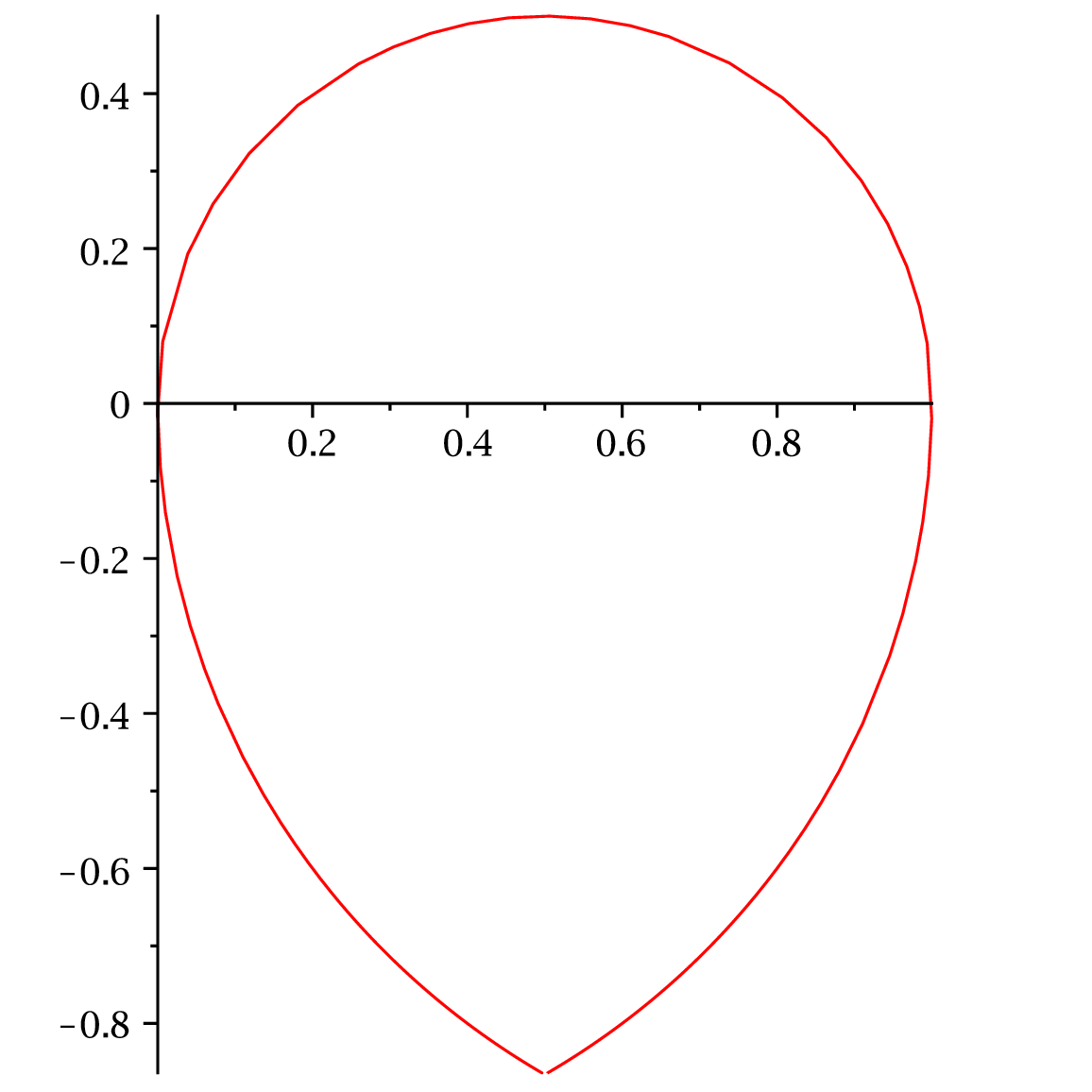}\ \  \includegraphics[scale=0.2]{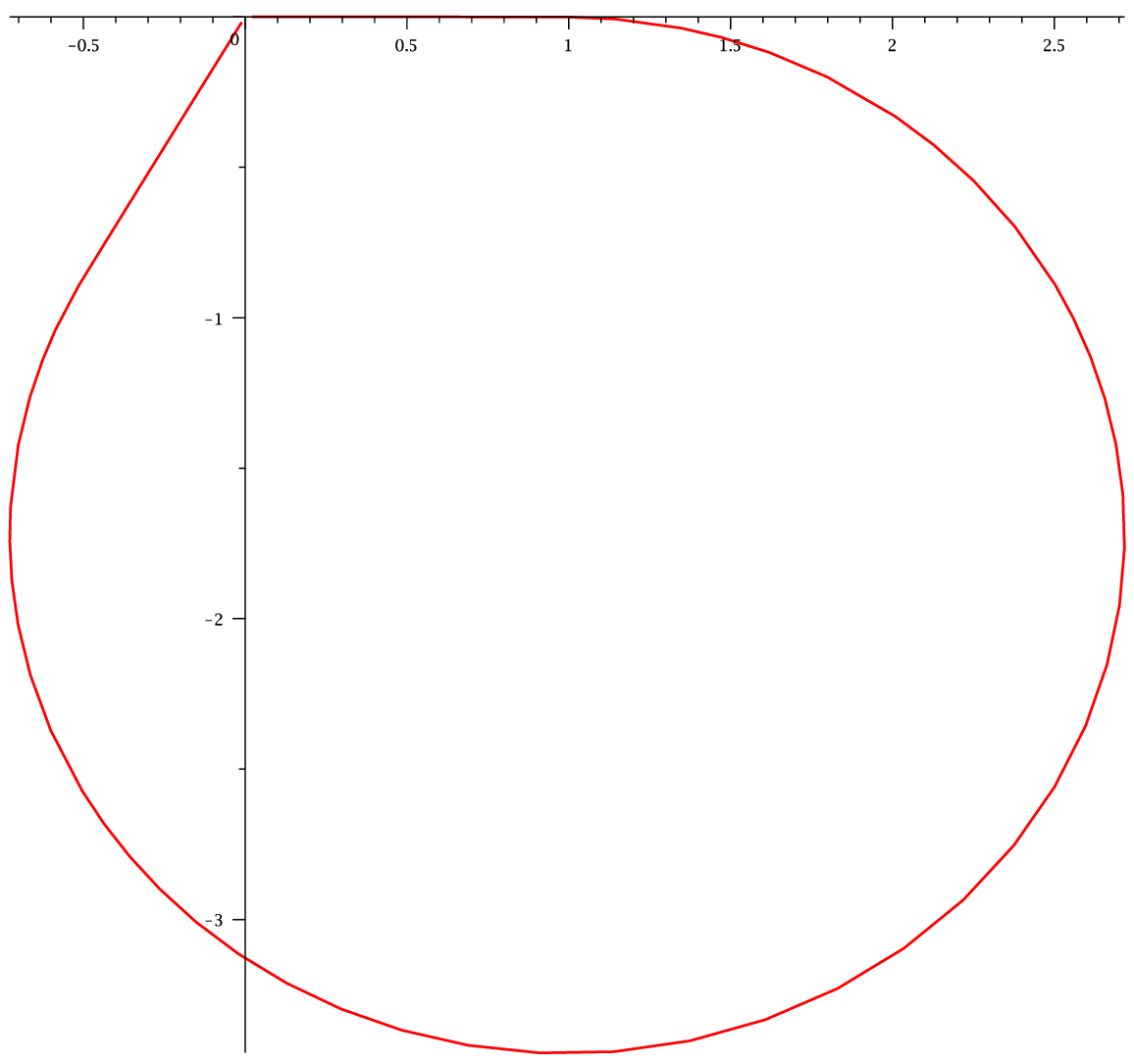}
  \caption{The image of $\HH$ under $\mathcal{K}$ ($f_a$) is the interior of the left
    (right, resp.) curve above.}\label{m}
  \label{fig:12}
\end{figure}
\begin{proof}
  \begin{enumerate}[(a)]
\item\label{itema} Since $t=\infty$ is a regular
singularity of \eqref{eq:hypg}, from the indicial
equation we see that there is a fundamental set of solutions of the form
\begin{equation}
  \label{eq:deff}
  f_2(s)=s^{5/6}A(1/s);\ \ f_1(s)=s^{1/6}B(1/s)
\end{equation}
where $A$ and $B$ are analytic and $A(0)=B(0)=1$. We choose the natural branch of the roots, where $s^{1/6}>0$ if $s>0$. Since \eqref{eq:hypg} has real coefficients, we can check that
$A$ and $B$ have real-valued Taylor coefficients.
\item
The differential equation for $g(z)=s^{5/6}J(-3/(4s));\ z=-3/(4s)$ is
\begin{equation}
  \label{eq:eqg}
  z(1-z)g''+\frac{1}{3}(1-z)g'+\frac{5}{36}g=0
\end{equation}
a  hypergeometric  equation in standard form. The solution analytic at $z=0$ is  $_2F_1(-\tfrac56,\tfrac16;\tfrac13,z)$,  \cite{Abramowitz} 15.5.1. Thus,
\begin{equation}
    \label{eq:eqj2}
    f_2(s)=s^{5/6}\,\,_2F_1(-\tfrac56,\tfrac16;\tfrac13,-\tfrac{4}{3s})
  \end{equation}
is the solution of \eqref{eq:hypg} with the properties in (\ref{itema}). By \cite{Abramowitz} 15.3.1,
\begin{equation}
  \label{eq:intrep}
  f_2(s)=\Gamma\(\frac13\)\Gamma\(\frac16\)^{-2}s^{5/6}\int_0^1 t^{-5/6}(1-t)^{-5/6}\(1+\frac{4t}{3s}\)^{5/6}dt
\end{equation}
Note that when $s$ is in the upper half plane
$\Im(1+{4t}/({3s}))^{5/6}<0$. Thus we have $s\in \HH\Rightarrow f_2(s)\ne 0$
and it follows that $f_a$ is analytic in $\HH$.

\item The function $f_1$ is given by
  \begin{equation}
  \label{eq:f1int}
  f_1=s^{1/6}\,\,\,_2F_1(-\tfrac16,\tfrac56;\tfrac53,-\tfrac{4}{3s})=\frac{2}{3}s^{1/6}\frac{\Gamma(\tfrac23)}{\Gamma(\tfrac56)^2}\int_0^1
t^{-\tfrac16}(1-t)^{-\tfrac16}(1+\tfrac{4t}{3s})^{\tfrac16}dt
\end{equation}
Indeed, \eqref{eq:f1int} solves \eqref{eq:hypg} and has the required behavior for large $s$.

\item  We note that $f_a=f_1/f_2$.  Note also that for $s>0$ or $s<-4/3$ the integrand is positive and
$f_1(s)$ does not vanish. The integrals in \eqref{eq:f1int} and
\eqref{eq:intrep} become elementary in the limit $s\to 0$ and we find that
\vspace{-0.25cm}\begin{equation}
  \label{eq:eqvals}
  a:=f_a(0)=\frac{1}{10}\frac{\Gamma \left(\frac{1}{6}\right)^3}{2^{2/3} 3^{1/3} \pi ^{3/2}}
\end{equation}\vspace{-0.25cm}
Similarly,  $f_a(-4/3)$ is also elementary and
\begin{equation}
  \label{eq:eqvals2}
  f_a(-4/3)=-e^{\frac{\pi i}{3}}a; \ \text{ and  also} \ \lim_{s\to\infty} f_a(s)=0
\end{equation}
\vspace{-0.25cm}
\item \label{ita} In fact, it follows just from Frobenius theory that
  neither $f_1$ nor $f_2$ are zero at $0$ or $-3/4$, thus they are not
  analytic there. Indeed, the
  indicial equation at $-3/4$ and $0$ has roots $0$ and $1$; the
  generic solution has a logarithmic singularity. If say, $f_1$ were
  analytic at zero, it is then analytic at $-3/4$ as well,
  else it would be single-valued. But if $f_1$ is analytic at $0$, the
  log singularity at $-3/4$ is also incompatible with the $5/6$
  branching at $\infty$. So $f_1$ is singular at both $1$ and $0$.
  Thus it is nonzero at those points (as the solution that vanishes
  say at $0$ corresponds to the solution $1$ of the indicial equation
  which is analytic by Frobenius theory).
\item \label{(g)}  Note that $(f_1/f_2)'=-W/f_2^2\ne 0$ where $W$ is the Wronskian of $f_1$ and $f_2$.
 Thus, $f_1/f_2$ maps $(0,\infty)$ one-to-one to the
  segment $(0,a)$ where $a=f_1(0)/f_2(0)$, see \eqref{eq:eqvals},  and $\infty$ is mapped to $0$.
\item \label{(k)} Similarly, by \eqref{eq:deff}, $f_1/f_2$ maps $(-\infty, 0)$
  one-to-one to the segment $e^{-4\pi i/3}(0,a)$, for the same $a$
by the symmetry implied by \eqref{eq:deff}, and, as before $-\infty$
mapped to zero.

\item \label{itb} By Frobenius theory and (\ref{ita}) above, both $f_1$
  and $f_2$ have a singularity which, to leading order, is of the form
  $t\ln t$ at $-\tfrac43$ ($0$, resp.), where $t=s+\tfrac43$ ($s$, resp.). Thus,
looking at the local mapping of a segment by $s\ln s$
near $s=0$, we see that at $0$ the angle change is $\pi$ and the orientation
of the arc is preserved. Thus,  the segment
  $[-\tfrac43,0]$ is mapped into a curve which is tangent to both $I_1$ and
  $I_2$, and the angle change at $-\tfrac43$ and $0$ is, in absolute value,
  $\pi$.

\item \label{(j)} We now determine this curve. Between $-\tfrac43$ and $0$ we take,
  once more relying on the real-valuedness of the coefficients at a
  different pair of independent solutions $f_3$ and $f_4$ which are
  real-valued and analytic on $[-\tfrac43,0]$. Then $f_3/f_4$ maps
  $[-\tfrac43,0]$ onto a segment on the real line. But since $f_1,...,f_4$
  solve \eqref{eq:hypg},  $f_3$ and $f_4$ are linear combinations
  with constant coefficients of $f_1$ and $f_2$. Thus $f_3/f_4$ is a
  M\"obius transformation of $f_1/f_2$, and the segment $[-\tfrac43,0]$ is
  mapped by $f_1/f_2$ onto a segment or an arccircle. Because the
  tangency showed in item (\ref{itb}), it must be an arccircle.

As in (\ref{(g)}) above, $(f_3/f_4)'\ne 0$ and thus $f_3/f_4$ is one-to
  one on $[-\tfrac43,0]$ and, since $f_a$ is a M\"obius transformation of
  $f_3/f_4$, $f_a$ is one-to-one as well on $[-\tfrac43,0]$.

Combining with (\ref{(g)}) and (\ref{(k)}) above, the image of $\HH$ is
int$(\mathcal{C}_1)$ and $f_a$ is one-to-one between $\RR$ and $\mathcal{C}_1$.

\item  By (\ref{(j)}) and  the argument principle, $f_a$ is one-to-one between $\HH$ and
int$(\mathcal{C}_1)$ (see also \cite{Lang} p. 227).

\item  If $|s|>5$ then, using the bound  $|\(1+\frac{4t}{3s}\)^{5/6}-1|<4/15$
in  \eqref{eq:intrep} it follows that
 $\left|f_2-1\right|<\frac{4}{15}$. Thus $|f_2(s)|>{7|s|^{5/6}}/{10}$.
Similarly, for $|s|>5$ we use \eqref{eq:f1int} to obtain $|f_1(s)|<{4|s|^{1/6}}/{5}$.
Therefore $f_a=f_1/f_2$ satisfies $|f_a(s)|<{8|s|^{-1}}/{7}<\frac{1}{4}$ if $|s|>5$.
  \end{enumerate}
\end{proof}

\begin{proof}[Proof of Proposition \ref{jjr}]
Since $f_a=f_1/f_2$ where $f_{1,2}$ solve \eqref{eq:difeq00}, as do $J$ and
$\hat{J}$, $\mathcal{K}$ is a linear fractional transformation of $f_a$, i.e.
\begin{equation}
  \label{eq:formK}
  \mathcal{K}(s)=a_1+\frac{a_2 f_a}{f_a+a_3}
\end{equation}
for some constants $a_{1,2,3}$ that we now determine. It follows from \eqref{eq:t21a} that
$\mathcal{K}(0)=0$ and $\mathcal{K}(-\tfrac43)=1$. For $\Re
s=0$ and $\Im s\to\infty$ the roots of $u^3/3+u^2+s$ approach the roots of
$u^3/3+s$, thus by Lemma \ref{lem1} we see that ${\ro}_2\sim i|3s|^{-1/3}$,
${\ro}_3\sim e^{-\pi i/6}|3s|^{-1/3}$, and ${\ro}_1\sim e^{-5\pi i/6}|3s|^{-1/3}$. It
then follows from \eqref{eq:t23} and \eqref{eq:t21} that
$\mathcal{K}(s)\to e^{-\pi i/3}$. The corresponding values for $f_a$
can be obtained directly using its definition since the hypergeometric functions can be calculated explicitly for $s=0,-\tfrac43,\infty i$ (see also \eqref{eq:eqvals}). We have $f_a(0)=a$, $f_a(-\tfrac43)=e^{4\pi i/3} a$, and $f_a(\infty i)=0$, which allows us to solve for $a_{1,2,3}$ and obtain
$$\mathcal{K}(s)=M(f_a(s));\ M(z):=e^{-\pi i/3}+{i \sqrt{3}z}\({z-e^{2\pi i/3}a}\)^{-1}$$
A simple calculation shows that the M\"obius transformation $M$ has the
properties: $M(I_1)$ is the arccircle through  $0$ and $e^{-\pi i/3}$,
tangent to $i\RR$, $M(I_2)$ is the
arccircle tangent to $i\RR$ passing through $e^{-\pi i/3}$ and $1$, and it maps the arccircle through  $a$ tangent at $z=e^{4\pi i/3}a$
to
  $\{z:\arg z= 4\pi i/3\}$ into the
arccircle  through $0$  tangent to $i\RR$ and to $1+i\RR$. It then follows from Lemma \ref{ffa}
that $\mathcal{K}$ maps the upper half plane into the interior of
$C_2$.

\z (ii) If $|s|>5$,  then by Lemma \ref{ffa}, $|f_a|<\tfrac14$ implying
$\left|i \sqrt{3}f_a\(f_a-e^{2\pi i/3}a\)^{-1}\right|
<\tfrac25 $ and thus using \eqref{eq:formK} we get $\Im
\mathcal{K}(s)<-\tfrac{\sqrt{3}}{2}+\tfrac{2}{5}<-\tfrac{2}{5}$.

Consider  $I$, the image under  $\mathcal{K}$ of region $\mathcal{\HH}\cap\{s:|s|>\eta_1\}\cap\{s:|\sm|>\eta_1\}$ and the compact set $I_{\epsilon}=\{z:\text{dist}(z,I_1)\Le \epsilon\}$. If $\epsilon$ is small enough then $I_{\epsilon}\subset \mathcal{K}(\HH)$ and thus dist$[\mathcal{K}^{-1}(I_\epsilon),\RR]$ is positive and increasing in $\epsilon$ implying \eqref{jje}.

(iii) This follows directly from Lemma \ref{eps1}.
\end{proof}

\section{Recurrence relations and Constants of Motion. Proof of Theorem\,\ref{recu1}(i)}\label{xjc}

We will use the integral equations \eqref{eq:eqds2n} and \eqref{eq:eqdx2n} to derive an asymptotic constant of motion formula for $x$ in the third quadrant.

\subsection{Notations}

\z (i) Denote $R_I(v)=R(v,s_I)$, ${J}_I(v)=J(v,s_I)$, etc., $R_n(v)=R(v,s_n)$, ${J}_n(v)=J(v,s_n)$, ${L}_n(v)=L(v,s_n)$, $J_n=J(s_n)$, $L_n=L(s_n)$ etc., and ${s_I}\h,^-=s_I+\frac 43$.

\z (ii) Denote, consistent with the notations of {Theorem}\,\ref{recu1},
\begin{equation}\label{defct}
    \mathcal{Q}=xJ,\ \ \ \   \mathcal{K}={\hat{J}}/{J},\ \ \ \  \mathcal{Q}_n={x_n}J_n,\ \   \mathcal{G}_n=\mathcal{Q}_n/x_0,\ \
  \mathcal{K}_n={\hat{J}_n}/{J_n}
\end{equation}

{\em{Note:}} since $  \mathcal{Q}_n$ is a large quantity (for large $x_0$) it is preferable to work with the normalized quantity $ \mathcal{G}_n$ which is $O(1)$.

\z (iii) We use $c_0,c_1,c_2\ldots$ to denote constants independent of $n,s_0,x_0,s_I,x_I$ etc., and
$\mathfrak{c}$ denotes a``generic''  such constant.

\z (iv) Consider the segment $\oldCa$ and its symmetric about the line $u=-1$, $\oldCb $, contained in $\mathcal{C}$:
\begin{equation}
  \label{eq:eqDab}
\oldCa
=\{t(1+i)/3:t\in[-1,1] \} \ \text{and}\  \oldCb
=\{-2+t(-1+i)/3:t\in[-1,1] \}
\end{equation}
($\oldCa$, $\oldCb$ are sub-segments of $[\frac{1+i}{3}, -(1+i)]$, respectively $[-(1+i),-2+\frac{-1+i}{3}]$, see Note\,\ref{chooseC}). Note that $u\in \oldCa$ if and only if $-2-\overline{u}\in \oldCb$, symmetry which will be used in the following, in conjunction with \eqref{sym_T} and with $|-2-\overline{u}|=|\um|$.

%\begin{multline}\label{sned0}
%s(u)=s_n-\frac{2J_n(u)}{x_n}-\frac{2}{x_n}\int_{u_n}^u \frac{s(v)-s_n}{R_n(v)+R(v,s(v))}dv\\+\frac{2}{x_n}\int_{u_n}^u \frac{R(v,s(v))L(v,s)}{x_n+L(v,s)}dv+\int_{u_n}^u \frac{784}{625(x_n+L(v,s))^4} dv \end{multline}
  %As long as \eqref{sned0} has a solution we can apply it to \eqref{eq:eqdx2n} to recover  a solution to the Painlev\'e equation P$_1$, see Note \ref{equiv}.

\subsection{Calculating the Poincar\'e map (the first return map)}\label{return_map}

The values of $s_{n},x_{n}$ are obtained by iterating the first return map. To establish its properties
consider \eqref{eq:eqds2n},\,\eqref{eq:eqdx2n} with initial conditions $(s_I,x_I)$:
\begin{align}
  \label{eq:eqds2I}
 & s(u)=s_I-2\int_{u_0}^u \(\frac{R(v,s(v))}{x(v)}-\frac{392}{625}\frac{1}{x(v)^4}\)dv \\
 & x(u)=x_I+\int_{u_0}^u \frac{1}{R(v,s(v))}dv \label{eq:eqdx2I}
\end{align}
Proposition\,\ref{prop1} shows that the system \eqref{eq:eqds2I},\, \eqref{eq:eqdx2I} has a unique solution for $u$ traveling once along $\mathcal{C}$ and for initial values $(s_I,x_I)\in\mathcal{R}$, a suitable region with the property that the final values  $(s_f,x_f)$ of $(s(u),x(u))$ when $u$ returns to $u_0$ are close to $(s_I,x_I)$, cf. \eqref{sfcsi},\,\eqref{xned}.

 Once this fact is proved, then $(s_n,x_n)$ are found by iterating the first return map
\begin{align}
  \label{first_rets}
 & s_f=\Phi(s_I,x_I)\equiv s_I-2\oint_ \mathcal{C}\(\frac{R(v,s(v))}{x(v)}-\frac{392}{625}\frac{1}{x(v)^4}\)dv \\
 & x_f=\Psi(s_I,x_I)\equiv x_I+\oint_\mathcal{C}\frac{1}{R(v,s(v))}dv \label{first_retx}
\end{align}
and $(s_n,x_n)=(\Phi,\Psi)^{\circ n} (s_0,x_0)$ for all $n=1,2,\ldots,N_m$ for which $(s_n,x_n)$ remain in $\mathcal{R}$. We calculate asymptotically the Poincar\'e
map \eqref{first_rets},\,\eqref{first_retx} in Lemma\,\ref{ass}, information needed to determine $N_m$ in {Proposition}\,\ref{Wholeinterval}.

\subsubsection{The region $\mathcal{R}$ of ininitial consitions $(s_I,x_I)$}\label{regRR12}
As explained in \S\ref{Assumption1}, $s_0$ starts near $0$, and $s_{N_m}$ ends near $-4/3$. The estimates must be worked out differently in these two regions in $s$\footnote{Of course, \eqref{sym_T} is a symmetry of the equation, but generally not of its particular solutions.}. Therefore we define $\mathcal{R}$ as a union $\mathcal{R}=\mathcal{R}_1\cup\mathcal{R}_2$, where $\mathcal{R}_1$ contains values of $s$ close to $0$ (but not very close) and far from $-4/3$, while  $\mathcal{R}_2$ has $s$ close to $-4/3$ (but not very close) and far from $0$; both contain intermediate values of $s$. They are defined as follows.

Let $\mathfrak{m}>0$ (cf. \eqref{ASSS1} and comments preceding it) be
large enough (independent of any other parameter), $\eta_1$ given by Lemma\,\ref{eps1}, and $\eta_2$ small so that \eqref{jje} holds. Recall that here $u_0=-4$ (though other values can also be used).

{\bf{Region $\mathcal{R}_1$}} {\em{is the set of all $(s,x)$ with: \newline
(i)  $|x|>\mathfrak{m}$, $|u_0^3/3+u_0^2+s|>\eta_2/2$, $s\in  \mathbb{D}_5^+,
  |\sm|>\eta_1/2$, $|s\,x|>1$ and
 \newline
 (ii) for all $ w\in\oldCa$ (see \eqref{eq:eqDab}) we have $s-2J(w,s)/x\in\mathbb{D}_5^+$
  and $ |s-2J(w,s)/x|>|s|/8$.}}

{\bf{Region $\mathcal{R}_2$}} {\em{is defined as $\mathcal{R}_1$, only with $s$ interchanged with $\sm$  and $\oldCa$ replaced by $\oldCb$.}}

{\em{Note}} that  $(s,x)\in\mathcal{R}_1$ if and only if $(\sm,x)\in\mathcal{R}_2$.

\subsubsection{Existence of the Poincar\'e map and estimates}

\begin{Proposition}\label{prop1} For $(s_I,x_I)$ in $\mathcal{R}$, the system  \eqref{eq:eqds2I},\, \eqref{eq:eqdx2I} has a unique
  solution $s(u),x(u)$ for $u$ going once along $\mathcal{C}$. Furthermore, with $\alpha$ is defined in \eqref{alp} the solution satisfies
  \begin{equation}\label{sfcsi}
  |s(u)-s_I+{2J_I(u)}/{x_I}|\Le 7\alpha\ln |x_I|/|x_I|^2
  \end{equation}

\end{Proposition}

The proof, given in \S\ref{PfProp1}, needs the estimates of \S\ref{EsSumsro} in the Appendix.
%%%%%%%%%%%%%%%%%%%%%%%%%%

%%%%%%%%%%%%%%%%%%%%%%%%%%

\subsubsection{Proof of {Proposition}\,\ref{prop1}}\label{PfProp1}
Due to the symmetry \eqref{sym_T}, the calculations with $s_I$ and $s_I^-$ are similar, see the
proof of Lemma \ref{m0} (iv), so we only
give the proof for $\mathcal{R}_1$.

\begin{proof}
 Inserting \eqref{eq:eqdx2I}  in \eqref{eq:eqds2I} and using the identity $\sqrt{a}-\sqrt{b}=(a-b)/(\sqrt{a}+\sqrt{b})$ we get
 \begin{multline}\label{sned0I}
s(u)=s_I-\frac{2J_I(u)}{x_I}-\frac{2}{x_I}\int_{u_0}^u \frac{s(v)-s_I}{R_I(v)+R(v,s(v))}dv\\+\frac{2}{x_I}\int_{u_0}^u \frac{R(v,s(v))L(v,s)}{x_I+L(v,s)}dv+\int_{u_0}^u \frac{784}{625(x_I+L(v,s))^4} dv
\end{multline}

Denoting $s(u)-s_I+2J_I(u)/x_I=\delta(u)$,  $R_I^-=R_I^-(v)=R(s_I-2J_I(v)/x_I+\delta(v)), L_I^-=L_I^-(v)=L(s_I-2J_I(v)/x_I+\delta(v))$
\eqref{sned0I} becomes
\begin{equation}\label{contr}{\text{\scalebox{0.95}{$
 \delta(u)=\int_{u_I}^u\left(\frac{2}{x_I
 ^2}  \frac{2J_I}{R_I+R_I^-}-\frac{2}{x_I}\frac{\delta}{R_I+R_I^-}
 +\frac{2}{x_I} \frac{R_I^-L_I^-}{x_I+L_I^-}
 +\frac{784}{625(x_I+L_I^-)^4} \right)=:\mathcal{N}(\delta)(u)$}}
}\end{equation}
%\begin{multline}\label{contr}
% \delta(u)=\frac{2}{x_I ^2}\int_{u_I}^u \frac{2J_I(v)dv}{R_I(v)+R_I^-(v)}-\frac{2}{x_I}\int_{u_I}^u \frac{\delta(v)dv}{R_I(v)+R_I^-(v)}\\
 %+\frac{2}{x_I}\int_{u_I}^u \frac{R_I^-(v)L_I^-(v)}{x_I+L_I^-(v)}dv +\frac{784}{625}\int_{u_I}^u\frac{dv}{(x_I+L_I^-(v))^4}=:\mathcal{N}(\delta)(u)
%\end{multline}
We show that  $\mathcal{N}$ is contractive in the ball $B=\{y:\sup\,|y(u)|\Le 7\alpha\ln |x_I|/|x_I|^2\}$ in the Banach space of continuous functions on one loop of the lifting of $\mathcal{C}$ on its universal covering, from $u_0$ to the lifting of $u_0$.

{\bf{Note.}} {\em{Once existence has been established, solutions turn out to be analytic in $u,s_I,x_I$ by standard  analytic ODEs theory.}}

We first check that the assumptions of Lemma \ref{m0} hold if
$\tilde{f}_{1}(v)=-2J_I(v)/x_I+\delta(v)$, namely that $s_I+\tilde{f}_{1}\in \newD6p$
and the inequalities in
\eqref{eq:eqf12} are satisfied if $|x_I|$ is sufficiently large. By assumption
$s_I-2J_I(v)/x_I\in\mathbb{D}_5^+$, we have $s_I-2J_I(v)/x_I+\delta(v)\in
\newD6p$ if $|\delta(v)|<\eta_2/2$. Since $ |s_I-2J_I(w)/x_I|>|s_I|/8$,
we have $|s_I+\tilde{f}_{1}(v)|>|s_I|/8-|\delta(v)|>|s_I|/10$ if
$|\delta(v)|<\frac{1}{40|x_I|}$. Finally $
|-2J_I(w)/x_I+\delta(v)|<2\alpha/|x_I|+|\delta(v)|<\eta_3/2$ if $|x_I|$ is
sufficiently large. Thus Lemma \ref{m0} applies.

 It follows that for
 large $\mathfrak{m}$ we have
$$\delta\in B\Rightarrow |\mathcal{N}(\delta)(u)|\Le \frac{6\alpha\ln
  |x_I|}{|x_I|^2}+\left|\frac{\mathfrak{c}\ln^2
    x_I}{x_I^3}\right|\,\,\forall u\in\mathcal{C}\Rightarrow \mathcal{N}(\delta)\in B$$
For $\delta_{1,2}\in B$, denoting $R^-_j(v)=R(v,s_I-2J_I(v)/x_I+\delta_j(v))$
and  $L^-_j(v)=L(v,s_I-2J_I(v)/x_I+\delta_j(v))$, $j=1,2$,  we have by Lemma \ref{m0}
\begin{equation}\label{r1r2}
|R^-_1(v)-R^-_2(v)|=\left|\frac{\delta_1-\delta_2}{R^-_1+R^-_2}\right|\Le\frac{\mathfrak{c}\sup|\delta_1-\delta_2|}{\sqrt{|v|^2+|s_I|}}
\end{equation}
and
\begin{equation}\label{r1r2n}
 \left|\frac{1}{R_I+R^-_1}-\frac{1}{R_I+R^-_2}\right|  =\left|\frac{\delta_1-\delta_2}{(R^-_1+R^-_2)(R_I+R^-_1)(R_I+R^-_2)}\right|\leqslant\frac{\mathfrak{c}\sup|\delta_1-\delta_2|}{(|v|^2+|s_I|)^{3/2}}
 \end{equation}
Thus
\begin{equation}\label{l1l2}
|L^-_1(v)-L^-_2(v)|
=\left|\int_{u_0}^v\frac{\delta_2(t)-\delta_1(t)}{R^-_1(t)R^-_2(t)
(R^-_1(t)+R^-_2(t))}dt\right|\leqslant \mathfrak{c}s_I^{-1}\sup_{\mathcal{C}}|\delta_2-\delta_1|
\end{equation}
 Using  \eqref{r1r2}, \eqref{r1r2n} and \eqref{l1l2} applied to \eqref{contr} it follows that, for
large $\mathfrak{m}$ we have
\\
\z $|\mathcal{N}(\delta_1)(u)-\mathcal{N}(\delta_2)(u)|\leqslant \mathfrak{c}x_I^{-1}\ln
x_I \sup_{\mathcal{C}}|\delta_2-\delta_1|\leqslant \frac{1}{2} $.

For initial data in $\mathcal{R}$, it follows from Proposition \ref{prop1} and Lemma \ref{m0} that
\begin{multline}\label{xned0}
x(u)=x_I+\int_{u_I}^u \frac{1}{R_I(v)}dv-\int_{u_I}^u\frac{s(v)-s_I}{R_I(v) R(v,s(v))(R(v,s(v))+R_I(v))}dv\\
\end{multline}\vspace{-0.25cm}
\z implying\vspace{-0.35cm}
\begin{equation}\label{xned}
|x(u)-x_I-{L}_I(u)|\leqslant \frac{\mathfrak{c}}{|x_Is_I|}
\end{equation}
and $|x(u)-x_I|\leqslant\mathfrak{c}(\ln |s_I|+1)$, since ${L}_I(u)\leqslant\mathfrak{c}(\ln |s_I|+1)$ by Lemma \ref{m0}. \end{proof}

\subsection{Calculating higher orders in the expansions of $x$ and $s$}
We can bootstrap Proposition \ref{prop1} in \eqref{eq:eqds2n}
and \eqref{eq:eqdx2n} to obtain, in principle, any number of terms in the
expansion of $s$ and $x$ for large $x_I$. Proposition\,\ref{fg01} and Lemma\,\ref{ass} give higher order terms in the asymptotic behavior of $s_f,x_f$ and of $\mathcal{Q}$, $\mathcal{K}$. These will be used in \S\ref{evol} to establish the number $N_m$ of times we can iterate the Poincar\'e map while preserving the asymptotic formulas.

%One quantity of particular interest is $\mathcal{Q}_I=x_IJ_I$. We have

\begin{Proposition}\label{fg01}

I. For $(s_I,x_I)\in\mathcal{R}_1$ the following estimates hold.

\z (i) The values of $s(u)$ for $u$ going once along $\mathcal{C}$ satisfy
\begin{multline}\label{sn2}
s(u)=s_I-\frac{2J_I(u)}{x_I}+\frac{2}{x_I^2}\int_{u_I}^u \frac{J_I(v)}{R_I(v)}dv+\frac{2}{x_I^2}\int_{u_I}^u R(v,s(v))\, {L}_I(v)dv+O(s_I^{-1}x_I^{-3})\\
=s_I-\frac{2J_I(u)}{x_I}+\frac{2J_I(u){L}_I(u)}{x_I^2}+O(x_I^{-3}s_I^{-1})
\end{multline}

\z(ii) The final value $x_f$ satisfies \vspace{-0.25cm}
\begin{equation}\label{xn2}
x_f=x_I+L_I-2G_I+O(x_I^{-1}s_I^{-1/2})
\end{equation} \vspace{-0.25cm}

\z(iii) $\mathcal{Q}=xJ$ is approximately constant (in the sense that $\mathcal{Q}=x_0J_0(1+o(1))$)and the first correction is
\begin{equation}\label{xje1}
\mathcal{Q}_f-\mathcal{Q}_I=x_fJ_f-x_{I}J_{I}=-2J_IG_I-2J_IF_I+O(x_I^{-1}s_I^{-1/2})
\end{equation}
whith the notation $F_I=F(s_I,x_I,J_I)$, $G_I=G(s_I,x_I)$ where
$$F(s,x,J)=\(\frac{sx}{4J}-\frac12\)
\Big(\ln(s)-\ln\left(s-{2 J}/{x}\right)\Big)-\frac12$$
$$G(s,x)=\frac12\left[\ln(s)-\ln(s-2J_{00}/x)\right]$$
where $J_{00}=J(0,0)$ cf. \eqref{eq:eqJLu},  is an elementary integral which evaluates to:
\begin{equation}
  \label{eq:eqj00}
  J_{00}=-\frac{12}{5}+\frac{2 (-2+u_0) (3+u_0)^{3/2}}{5 \sqrt{3}}=\frac{-12+4\sqrt{3}i}{5} \ \ ({\rm{since\ }} u_0=-4)
\end{equation}

\z(iv)  The functions $F_I$ and $G_I$ are
    $O(x_I^{-1}s_I^{-1})$.

II. For $(s_I,x_I)\in\mathcal{R}_2$ similar statements hold after replacing
$u,s,s_I$ by $\um,\sm,{s_I}\h,^-$ (where ${s_I}^-=s_I+\frac 43$), namely
\begin{equation}\label{xje2}
\mathcal{Q}_f-\mathcal{Q}_I=-2J_I\tilde{G}_I-2J_I\tilde{F}_I+
O\({x_I^{-1}({s_I}\h,^-})^{-1/2}\)
\end{equation}
where $\tilde{F}_I=\tilde{F}(s_I,x_I,J_I)$, $\tilde{G}_I=\tilde{G}(s_I,x_I)$
and
\begin{equation}
\tilde{F}(s,x,J)=\(\frac{i\sm x}{4J}-\frac{i}{2}\) \left[\ln (\sm)-\ln \left(\sm-2
  J/x\right)\right]-\frac{i}{2}
\end{equation}
\begin{equation}
\tilde{G}=\tilde{G}(s,x)=\frac{i}{2}  \left[\ln \sm-\ln \left(\sm -2\, J(-2,-\tfrac43)\,/{x}\right)\right]
\end{equation}
where  cf. \eqref{eq:eqJLu}
$J(-2,-\tfrac43)=-\frac{12 i}{5}$.

The functions $\tilde{F}_I$ and $\tilde{G}_I$ are $O(x_I^{-1}({s_I}\h,^-)^{-1})$.

\end{Proposition}

\begin{proof}

It suffices to prove I., then II. follows due to the symmetry \eqref{sym_T}, cf. the
proof of Lemma \ref{m0} (iv).

(i)  We have, with the notation $\epsilon_I=|v|+\sqrt{|s_I|}$,
\begin{equation}\label{rsrn0}
\frac{1}{R(v,s(v))+R_I(v)}-\frac{1}{2R_I(v)}=\frac{s_I-s(v)}{2R_I(v)(R(v,s(v))+R_I(v))^2}
=O\({\epsilon_I^{-3}}x_I^{-1}\)
\end{equation}
where we used Lemma \ref{m0} in the last equality. Thus Proposition
\ref{prop1} and \eqref{rsrn0} imply
\begin{multline}\label{rsrn}
R(v,s(v))=R_I(v)+\frac{s(v)-s_I}{R(v,s(v))+R_I(v)}\\
=R_I(v)-\frac{J_I(v)}{x_IR_I(v)}+O\(x_I^{-2}\epsilon^{-1}_I(\epsilon_I^{-2}+|\ln x_I|)\)
\end{multline}

Using \eqref{rsrn0} and \eqref{rsrn} we can rewrite \eqref{sned0I} as \eqref{sn2}.

(ii)
To improve the estimate for $x(u)$, we denote $W_I=
\sqrt{R_I^2-\frac{2J_I}{x_I}}$ and note that \eqref{sn2} implies\vspace{-0.25cm}
$$R(v,s(v))-W_I(v)=O\(\frac{\ln s_I}{\epsilon_I^2x_I^2}\) $$\vspace{-0.25cm} and thus
\begin{equation}\label{ssr}
\frac{s(v)-s_I}{R_I(v)R(v,s(v))(R(v,s(v))+R_I(v))}=\frac{-2J_I(v)}{x_IR_I(v)W_I(v)(W_I(v)+R_I(v))}+O\(\frac{\ln s_I}{\epsilon_I^2x_I^2}\)
\end{equation}
Let $\rho_I=\sqrt{v^2+s_I}$ and
$\tilde{\rho}_I=\sqrt{v^2+s_I-2J_{00}/x_I}$. To simplify the estimate in \eqref{ssr}, we use the fact that
$J(v,s_I)-J(v,0)=O(s_I\ln s_I)$ and $J(v,0)-J_{00}=O(v)$ which together with Lemma \ref{m0} implies
\begin{equation}\label{rwrw}\frac{s(v)-s_I}{R_I(v)R(v,s(v))(R(v,s(v))+R_I(v))}=
\frac{-2J_{00}}{x_I\rho_I\tilde{\rho}_I(\rho_I+\tilde{\rho}_I)}
+O\(\epsilon_I^{-2}x^{-1}_I\)
\end{equation}
The function on the right side of \eqref{rwrw} can be integrated explicitly:
\begin{equation}\label{rwrw1}
\int_{u_I}^{0}\frac{-2J_{00}(\rho_I+\tilde{\rho}_I)dv}{x_I\rho_I\tilde{\rho}_I}
=-G_I(s)+O(x_I^{-1})\ \ \ \text{where }\ \
2G_I:=\ln(s_I)-\ln(s_I-2J_{00}/x_I)
\end{equation}
We apply \eqref{rwrw} and \eqref{rwrw1} to \eqref{xned0} and obtain the recurrence relation \eqref{xn2}.

(iii) To obtain the change in $J$  we use the definition of $J_I$ as well as \eqref{sn2} to get
\begin{multline}\label{jn1}
J_f-J_I=\oint\frac{s_f-s_I}{R_I+R_f}dv=-\frac{2J_I}{x_I}\oint\frac{1}{(R_I+R_f)}dv+O(x_I^{-2}\ln s_I)\\
=-\frac{2J_I}{x_I}\(\frac{L_I}{2}-\oint\frac{s_f-s_I}{2R_I(R_I+R_f)^2}dv\)+O(x_I^{-2}\ln s_I)\\
=-\frac{2J_I}{x_I}\(\frac{L_I}{2}+\frac{J_I}{x_I}\oint\frac{1}{R_I(R_I+R_f)^2}dv\)+O(x_I^{-2}\ln s_I)
\end{multline}
Using \eqref{sn2} and Lemma \ref{m0} with $\tilde{f}_1(v)=s_f$ we get
\begin{equation}\label{rr2}
\frac{1}{R_I(v)(R_I(v)+R_f(v))^2}=\frac{1}{\rho_I(\rho_I+\sqrt{v^2+s_I-2J_I/x_I})^2}+O\(\epsilon_I^{-2}\)
\end{equation}
Since the function on the right hand side of \eqref{rr2} can be integrated explicitly, we have

\begin{equation}
\oint\frac{dv}{\rho_I(\rho_I+\sqrt{v^2+s_I-2J_I/x_I})^2}
=\frac{x_IF_I(s_I,x_I,J_I)}{J_I}+O\(s_I^{-1/2}\)
\end{equation}
Thus
\begin{equation}\label{jf}
\frac{J_I}{x_I}\oint\frac{1}{R_I(R_I+R_f)^2}dv=F_I+O\(x_I^{-1}s_I^{-1/2}\)
\end{equation}
Applying \eqref{jf} to \eqref{jn1} we get
\begin{equation}\label{jn2}
  J_f=J_I-\frac{L_IJ_I}{x_I}-\frac{2J_IF_I}{x_I}+O\(x_I^{-2}s_I^{-1/2}\)
\end{equation}
The conclusion then follows from a straightforward calculation using \eqref{xn2} and \eqref{jn2}.

(iv) When
    $|s_I|>4\max(|J_I|,|J_{00}|)/|x_I|$ the
    $\ln$ in the expressions of $F_I$ and $G_I$ can be Taylor-expanded while
    if $1/|x_I|<|s_I|\le 4\max(|J_I|,|J_{00}|)/|x_I|$, then $F_I,G_I=O(1)$ by
    straightforward estimates using the definitions of $F_I$ and $G_I$.

\end{proof}

\begin{Lemma} \label{ass}
If $(s_I,x_I)\in\mathcal{R}$
there exists a constant $\oldDa>1$ so that
\begin{equation}\label{sna}
|s_f-s_I+{2J_I}/x_I|<\oldDa|x_I^{-2}|(|\ln s_I|+|\ln {s_I}\h,^-|)
\end{equation}
\begin{equation}\label{xna}
|x_f-x_I-L_I|<\oldDa |x_I^{-1}|\,\Big(|s_I|^{-1}+|{s_I}\h,^-|^{-1}\Big)
\end{equation}
\begin{equation}\label{jna}
|J_f-J_I+L_IJ_I/x_I|<\oldDa|x_I^{-2}|\,\Big(|s_I|^{-1}+|{s_I}\h,^-|^{-1}\Big)
\end{equation}
\begin{equation}\label{xjna}
|\mathcal{Q}_f-\mathcal{Q}_I|=|x_fJ_f-x_IJ_I|<\oldDa|x_I^{-1}|\Big(|s_I|^{-1}+|{s_I}\h,^-|^{-1}\Big)
\end{equation}
\begin{equation}\label{ksn1}
\big| \mathcal{K}_{f}-\mathcal{K}_I-\epsilon_0\,\frac{\mathcal{Q}_0}{\mathcal{Q}_I}\big|<\oldDa\,|x_I|^{-2}
\left(|\ln
s_I|+|\ln {s_I}\h,^-|\right)
\end{equation}
where $\epsilon_0=-2\pi i/x_0$.
\end{Lemma}
\begin{proof}
The first four estimates follow directly from Lemma \ref{m0} and \eqref{xned}, \eqref{sn2}, \eqref{jn1}.

To show \eqref{ksn1} we note that
\begin{equation}
  \label{eq:eqj21}
  \mathcal{K}'=\left(\frac{\hat{J}}{J}\right)'=\frac{J\hat{L}-\hat{J}L}{2J^2}=\frac{c_0}{2J^2},\ \ \ \mathcal{K}''=-\frac{c_0L}{2J^3},\ \ c_0=-48\pi i/5
\end{equation}
where we used the fact that \eqref{eq:difeq00} has no first derivative term,
and the value of $c_0$ follows from \eqref{eq:t21a}, Proposition  \ref{js0} and \eqref{eq:id3}.

By \eqref{alp} and Lemma \ref{bet}
we have
 $ \beta\Le|J(s_I)|\Le \alpha$.

Frobenius theory applied to \eqref{eq:difeq01} shows that $L(s-s_I)=O(\ln(s-s_I))$ for any singular point $s_I$ (that is, $s_I\in \{0,-4/3\}$).
Thus Taylor's theorem, \eqref{eq:eqj21}, and \eqref{sna} imply
$$\mathcal{K}(s_{f})-\mathcal{K}(s_{I})=\frac{48\pi i}{5\mathcal{Q}_I}+O\left(x_I^{-2}\left(|\ln s_I|+|\ln s_{I}^-|\right)\right)$$ which together with Lemma \ref{eps1} leads to \eqref{ksn1} (with $\mathcal{K}(s_{f})=\mathcal{K}_f$ etc.)
\end{proof}

\subsection{The solution of \eqref{eq:eqds2n},
  \eqref{eq:eqdx2n} exists for $n$ large enough so that $x_0,x_1,\ldots,x_{N_m}$ traverse the sector from edge to edge}  \label{evol}
The main results in this section are Proposition \ref{Wholeinterval}  and Corollary\,\ref{finc}, which are proved under Assumption \eqref{ASSS1}, which implies $(s_0,x_0)\in\mathcal{R}_1$. It will turn out that the iteration ends in $\mathcal{R}_2$.

%By Proposition \ref{jjr}, $s_0\in\HH$ implies $\mathcal{K}(s_0)\in\text{int}\,\,\mathcal{C}$.

Denote
\vspace{-0.25cm}\begin{equation}
  \label{eq:defep}
  \epsilon_0=-2\pi i/x_0,\ \ \epsilon_+=|\epsilon_0|
\end{equation}\vspace{-0.25cm}
We note that
\begin{equation}
  \label{eq:ep0epp}
  |\epsilon_0/\epsilon_+-1|\Le 2\epsilon_+|\ln\epsilon_+|
\end{equation}
by \eqref{ASSS1}, for large enough $\mathfrak{m}$. Let
\begin{equation}
  \label{defNs}
  N_s=\lfloor(\epsilon_+)^{-1}-(\epsilon_+)^{-1/2}\rfloor,\ \ \ \ \ \ \
  j_-=(\epsilon^+)^{-1}-j\ \ {\rm{for}}\ \ 0\Le j\Le N_s
  \end{equation}

  \begin{Proposition}\label{Wholeinterval}
Consider $s_0,x_0$ satisfying \eqref{ASSS1}. Then there exists $N_{m}>N_s$ so that the solution of the integral equations
\eqref{eq:eqds2n},\,\eqref{eq:eqdx2n} where we take $n=0$ exists along $\mathcal{C}$ for $N_{m}$
  loops and so that we have
  $$0<\Im s_{N_{m}}<11\, |x_0|^{-1}\ \ \text{and}\ \ \ \  |\Re s_{N_{m}}+\tfrac43|<2|x_0|^{-1/2}$$
\end{Proposition}

The proof of Proposition\,\ref{Wholeinterval} is given in \S\ref{ires} for going along $\mathcal{C}$ the first $n\leq N_s$ loops, followed by \S\ref{PfWholeInt} for a number $N_s<n<N_m$ loops.

 \subsubsection{Iteration of the Poincar\'e map a number of $n\leq N_s$ times} \label{ires}

 While $\mathcal{Q}_n$ and  $\mathcal{K}_n$ change from $n$ to $n+1$ by a term much smaller than their order, when expressed in terms of $\mathcal{Q}_0$ and  $\mathcal{K}_0$ the sum of the corrections is not small enough; in this section we show that more accurate expressions for the  discrete asymptotic conserved quantities are
  \begin{multline}
    \label{eq:basicineq}
    \tilde{\mathcal{Q}}_j=\mathcal{Q}(s_0)+a_j\ln (j+1)j_-\,\,;\ \ \tilde{\mathcal{K}}_j=\mathcal{K}(s_0)+j\epsilon_++b_j\epsilon_+^2 j\ln [(j+1)j_-],\,\,\,
  \end{multline}
 where $a_j,b_j$ may depend of $s_0,x_0$, but are bounded by a constant $ \oldC0$ independent of $s_0,x_0$:
 \begin{equation}\label{estajbj}
 |a_j|,\, |b_j|\leqslant \oldC0\ \ \text{for all }j=0,1,\ldots,N_s\ \text{ for some }\ \oldC0\ \ \text{large enough}
 \end{equation}
 proved in Proposition\,\ref{indc}, with the help of Lemma\, \ref{Consistency1}, by complete induction on $n$.

Then
\begin{equation}\label{deftilde}
  \tilde{s}_j={\mathcal{K}}^{-1}(\tilde{\mathcal{K}}_j),\ \ \  \tilde{x}_j=\tilde{\mathcal{Q}}_j/J(\tilde{s}_j)
\end{equation}
are expected to be the leading order of $s_j, x_j$, as expected form \eqref{eq:eq44d1},\,\eqref{eq:q3}.

\begin{Note}\label{sym1} {\rm As before denote $\sm=s+4/3,
    \stjm= \tilde{s}_j+4/3$ and so on. We only prove the results in
    this section for $\tilde{s}_j$, and the proofs for $\tilde{s}_j+\frac43 $
    are similar due to the symmetry \eqref{sym_T}, see the proof of Lemma
    \ref{lem1} (iii).

    We write $o(1)$ for quantities that vanish as $x_0\to \infty$ (therefore as
    $\epsilon_0\to 0$).}
 \end{Note}

\begin{Lemma}\label{Consistency1}
    Let $n\Le N_s$. Assume
  \eqref{eq:basicineq},\,\eqref{estajbj} hold for all $j=0,1,\ldots,n$
  and $\oldC0$ is large enough. Then there exist two
  constants $\oc1,\oldc2>0$, independent of $\oldC0, s_0,x_0$, such that for
  all $1\Le j\Le n$ and large  $x_0$ we have
\begin{enumerate}[(i)]
 \item \ \  $\tilde{s}_j\in\HH$ and  $|\tilde{s}_j|<5$
\item  \ \
$\oc1 j\epsilon_+\Le|\tilde{s}_j|\Le \oldc2 j\epsilon_+,\ \ \ \ \ \oc1(1-j\epsilon_+)|\Le\, |\stjm |\Le \oldc2
(1-j\epsilon_+)$
\item  \ \ $\frac12\,\frac{\alpha}{\beta} \Le |\tilde{x}_j/x_0|\Le 2\,\frac{\alpha}{\beta}$ where $\alpha,\beta$ are given by Lemma\,\ref{bet} and \eqref{alp}.
  \end{enumerate}
\end{Lemma}

\begin{proof}
 \eqref{eq:basicineq} implies
  $|\mathcal{K}(s_j)-j\epsilon_+|=O(j\epsilon_+^2\ln\epsilon^{-1}_+)$ for large $x_0$ and  all
  $j\Le N_s$. Thus $\mathcal{K}_j$ traverses $[0,1]$ up to small
  corrections.
  \begin{enumerate}[(i)]

  \item \label{it22} By the above, $\Im
    \mathcal{K}(s_j)=O(j\epsilon_+^2\ln\epsilon^{-1}_+)$. Now Proposition
    \ref{jjr} (ii) implies $\tilde{s}_j\in\HH$ and $|\tilde{s}_j|<5$. Lemma
    \ref{bet} and \eqref{alp} now give
\begin{equation}\label{new43}
\beta\Le|J(\tilde{s}_j)|\Le \alpha
\end{equation}
\item \label{it23} For small $w$,  Proposition \ref{jjr} (iii) implies
\begin{equation}
  \label{eq:cKinv}
  \left|\mathcal{K}^{-1}(t)-\tfrac{24i}{5\pi}t\right|\Le |t^{3/2}|;\ \  \left|\mathcal{K}^{-1}(1-t)+\tfrac{4}{3}-\tfrac{24i}{5\pi}t\right|\Le |t^{3/2}|
\end{equation}
and thus $\tilde{s}_j/(j\epsilon_+)$ is bounded above and below when
$j\epsilon_+,\,\,(j>0)$ is small. The rest is immediate.

\item \label{it24}  This follows by straightforward estimates from
  \eqref{eq:basicineq} and \eqref{new43}.
\end{enumerate}
\end{proof}

\begin{Proposition}\label{iterworks} Let $n$ be such that the assumptions of Lemma\,\ref{Consistency1} hold. If $s_0,x_0$ satisfy \eqref{ASSS1} with $\mathfrak{m}$ large enough, then $(\tilde{s}_j,\tilde{x}_j)$ (defined in \eqref{deftilde}) belong to
$\mathcal{R}$ (defined in \S\ref{regRR12}) for all $j=0,1,\ldots,n$.
\end{Proposition}

\begin{proof}
By Note \ref{sym1}, it suffices to look at those $j$ for which
$|\stjm|>\eta_1$.
Lemma \ref{Consistency1} implies  $|\tilde{x}_j|\Ge\frac{\alpha}{2\beta}|x_0|$ and
$\tilde{s}_j\in\mathbb{D}_5^+$.
 With $u_0=-4$ we have $|u_0^3/3+u_0^2+\tilde{s}_j|>\tfrac{16}{3}-5$  by  Lemma \ref{Consistency1} (i).

 The property $|\tilde{s}_j \tilde{x}_j|>1$  only  needs to be checked when $\tilde{s}_j=o(1)$, by Lemma
 \ref{Consistency1} (ii) and (iii).   That is, by
 Proposition \ref{jjr}, we look at those $j$ for which  $j\epsilon_+=o(1)$. In this case, by
 \eqref{eq:basicineq} we have
 \begin{equation}
   \label{eq:xJJ}
   J(\tilde{s})=J(0)(1+o(1))\Rightarrow \tilde{x}_j=x_0(1+o(1))
 \end{equation}
 and the rest
 follows from the definition of $\tilde{s}_j$, \eqref{eq:basicineq} and
 \eqref{eq:cKinv}.

Using the definition of $J$ we have $J(u,s)-J(u,0)\to 0$ as $s\to 0$, and $J(u,0)$ is
given by an
elementary  integral.
 Let $J_T=-\frac{12}{5}+\frac{u^2}{2}+\frac{4 i \sqrt{3}}{5}$ be the  two term-Taylor
expansion of $J(u,0)$; the Taylor remainder $|J(u,0)-J_T(u)|$ is bounded by
$1/10$ for $u\in\ell$. Using this bound and calculating  $\Re
J_T,\Im J_T$  for $u\in \ell$
we get
\begin{equation}
  \label{eq:inq1}
  -16/5<\Re J(u,0)<-2\ \ \ \text{and} \ \ \ \ 2/3<\Im J(u,0)<11/5
\end{equation}
For small $s$ \eqref{eq:inq1} and \eqref{eq:xJJ} imply
$$\Re \tilde{s}_j+\tfrac1{|\tilde{x}_j|}<\Re \(\tilde{s}_j-\tfrac{2J(u,\tilde{s}_j)}{\tilde{x}_j}\)<\Re
\tilde{s}_j+\tfrac5{|\tilde{x}_j|};\ \ \ \Im \left( \tilde{s}_j-2\tfrac{J_I(u)}{\tilde{x}_j}\right)>\Im \tilde{s}_j+\tfrac3{|\tilde{x}_j|}$$
Since $\tilde{s}_j\in\HH$, using these inequalities, we see that $\tilde{s}_j-2J_I(u)/\tilde{x}_j\in \mathbb{D}_5^+$ and
$$|\tilde{s}_j-2J_I(u)/\tilde{x}_j| \Ge \max\left\{\Re \tilde{s}_j-5/|\tilde{x}_j|,\ |\Im \tilde{s}_j|+3/|\tilde{x}_j|\right\}>|\tilde{s}_j|/8$$
for $u\in\ell$. For $n$ close to $N_s$  see Note \ref{sym1}.
\end{proof}

\begin{Proposition}[The evolution ``preserves'' \eqref{eq:basicineq}] \label{indc} \

 Let $n \Le N_s$.
Assume \eqref{eq:basicineq},\,\eqref{estajbj} are true for $j=1,...,n-1$.

Consider the initial conditions $(s_I,x_I)=(\tilde{s}_{n-1},\tilde{x}_{n-1})$. By Propositions
\ref{iterworks} and \ref{prop1},  the solution exists for one more loop and $(s_f,x_f)$ are well defined, if $\mathfrak{m}$ is large enough.

Then with  $\tilde{\mathcal{Q}}_n=x_fJ_f$ and
$\tilde{\mathcal{K}}_n=\mathcal{K}(s_f)$, $\tilde{\mathcal{Q}}_j,\tilde{\mathcal{K}}_j$  satisfy
\eqref{eq:basicineq}, \eqref{estajbj} for  $j=1,...,n$ for some $c_2>0$.
\end{Proposition}
\begin{proof}[Proof of Proposition \ref{indc}]
 Using Lemma \ref{Consistency1} to estimate $1/x_f$, $1/s_f$, we see that for large $x_0$
there is a constant  $\ogc1$ (independent of $n$, $\oldC0,s_0,x_0$) so that
  \begin{equation}
    \label{eq:firstineq}
    |x_f^{-1}|(|s_f|^{-1}+|s_f^-|^{-1})\Le \ogc1 (1/
    n+1/N_s)
  \end{equation}
Now, \eqref{xjna} and \eqref{eq:firstineq} imply
\begin{equation}
  \label{eq:qjp1}
 | \tilde{\mathcal{Q}}_{n}-\tilde{\mathcal{Q}}_{n-1}|\leqslant \oldDa\ogc1\left(\frac{1}{n-1}+\frac{1}{({n-1})_-}\right)
\end{equation}
implying that \eqref{eq:basicineq},\,\eqref{estajbj}
hold for $\tilde{\mathcal{Q}}_j$ for all $j=1,...,n$ if $\oldC0>2\oldDa\ogc1$.  This fact, and \eqref{eq:ep0epp}, and Lemma \ref {Consistency1},
used to estimate $\tilde{x}_j$ and $\tilde{s}_j$ in \eqref{ksn1}, show that
\begin{multline}\label{136n}
|\mathcal{K}(\tilde{s}_{n})-\mathcal{K}(\tilde{s}_{n-1})-\epsilon_+|<\tfrac15\oldC0\epsilon_+^2\ln [(j+1)j_-] +2\epsilon_+^2|\ln \epsilon_+|\\
+\frac{4\oldDa \beta^2}{\alpha^2|x_0|^{2}}(|\ln[(j+1)j_-]+2|\ln (\oc1\epsilon_+)|)
\end{multline}
Adding the errors in \eqref{136n}, and using the fact that $\epsilon_+^{-1}<2(j+1)j_-$ it follows that
$\tilde{\mathcal{K}}_j$ satisfy \eqref{eq:basicineq}, \eqref{estajbj} for all $j=1,...,n$.
\end{proof}

We can now obtain estimates for $\mathcal{Q}_n$ and $\mathcal{K}_n$:

\begin{Corollary}[Inductive construction of the solution of \eqref{eq:eqds2n},
  \eqref{eq:eqdx2n}]\label{finc} The solution of the integral equations \eqref{eq:eqds2n},
  \eqref{eq:eqdx2n} exists along $\mathcal{C}$ for $N_s$
  loops. Furthermore,
\begin{equation}\label{sns}
|s_{N_s}\,^--4i|x_0|^{-1/2}|<|x_0|^{-1/2};\;  |\frac{x_{N_s}}{x_0}+i|<(\oldC0+1)\ln |x_0|/|x_0|
\end{equation}
In particular $(s_{N_s},x_{N_s})\in\mathcal{R}_2$.
  \end{Corollary}
  \begin{proof}
    With $s_I=s_0, x_I=x_0$, we get by Proposition \ref{prop1}
    $x_1=x_f,s_1=s_f$, and \eqref{eq:basicineq} follow from \eqref{xjna} and
    \eqref{ksn1}. Thus, by Proposition \ref{iterworks}, Proposition
    \ref{prop1} applies, to yield $\mathcal{Q}_2$ and $\mathcal{K}_2$ which by
    Proposition \ref{indc} satisfy \eqref{eq:basicineq} and, inductively
$x_j,s_j$ yield $\mathcal{Q}_j,\mathcal{K}_j$ %not tilde
for all $j\Le N_s$.

The estimate for $s_{N_s}$ follows from \eqref{eq:cKinv}, and the estimate for $x_{N_s}$ follows from Lemma \ref{eps1} and \eqref{eq:basicineq}.
  \end{proof}

\subsubsection{Proof of Proposition\,\ref{Wholeinterval} for $n>N_s$ up to $n=N_m$}\label{PfWholeInt}

We prove by complete induction that Proposition \ref{prop1} applies to $s_n,x_n$ with
$n\Ge N_s-1$ as long as $\Im s_n\Ge 11/|x_0|$.

First note that Proposition \ref{prop1} applies to $s_{N_s-1},x_{N_s-1}$ by Corollary \ref{finc}.
Suppose for some $N_s\Le n<N_s+|x_0|^{1/2}$
 we have that
$(s_k,x_k)\in\mathcal{R}_2$, $|s_{k}\,^-|<\eta_1$, and $\Im s_k\Ge 11/|x_0|$
for all $k$ with $N_s-1\Le k< n$. We only need to verify the following conditions defining $\mathcal{R}_2$:  $|x_n|>\mathfrak{m}$, $s_n\in  \mathbb{D}_5^+$, $|x_ns_{n}\,^-|>1$ and that for all $ w\in\ell^-$  we have $s_n-2J_n(w)/x_n\in\mathbb{D}_5^+$  and $ |s_{n}\,^- -2J_n(w)/x_n|>|s_{n}\,^-|/8$, since the other conditions are obvious.

By \eqref{xna}, \eqref{sna}, Lemma \ref{eps1} we have $|x_k-x_{k-1}|<\olgoc2|\ln x_0|$  for some $\olgoc2$, $|s_{k}\,^-|<8(n-k)/|x_0|$, and
\begin{equation}\label{sksk}
\left|s_k-s_{k-1}+\frac{48}{5}\frac{1}{x_0}\right|<\frac{1}{|x_0|}
\end{equation}
for $N_s-1\Le k \Le n$.
Thus by \eqref{sns} we have $|s_{n}\,^-|<|s_{N_s}\,^-|+11|x_0|^{-1/2}<16|x_0|^{-1/2}$ and
 $$ |{x_{n}}/{x_0}+i|<|{x_{N_s}}/{x_0}+i|+\olgoc2|x_0|^{-1/2}|\ln
x_0|<\(\olgoc2+1\)|x_0|^{-1/2}|\ln x_0|$$
which implies
\begin{equation}\label{eq139}
 |s_{n}\,^-|<\eta_1,\ \ \text{and}\ \ \ |x_n|=|x_0|(1+o(1))
\end{equation}
and by \eqref{sksk}
$$\Im s_n>\Im \(s_{n-1}+\frac{48}{5x_0}\)-\frac{1}{|x_0|}>0 $$
Thus $|x_n|>9|x_0|/10>\mathfrak{m}$, $s_n\in  \mathbb{D}_5^+$, and $|x_ns_{n}\h,^-|>9|x_0\Im s_n|/10>1$.

A calculation similar to that used in the proof of Proposition \ref{iterworks} shows that
$$|\Re J(w,s_n)|<1\ \ \text{and}\ \  -13/4<\Im J(w,s_n)<-8/5\ \ \ \ \ \ \ \text{for }w\in \oldCb$$
Thus $\Im (s_{n}\h,^- -2J_n(u)/x_n)>\Im s_n-7/|x_n|>0$ and
$$|s_{n}\h,^- -2J_n(w)/x_n|\Ge \max(|\Re s_{n}\h,^-|-3/|x_n|, \Im s_n-7/|x_n|)>|s_{n}\h,^-|/8$$
Thus $s_{n},x_{n}$ are in Region 2, and Proposition \ref{prop1} applies again.

Since $\Im
s_{N_s}<5|x_0|^{-1/2}$ by \eqref{sns}  and $\Im (s_k-s_{k-1})<-8/|x_0|$ by \eqref{sksk}, there must exist some $N_{m}<N_s+|x_0|^{1/2}$ such that $0<\Im s_{N_{m}}<11/|x_0|$. By \eqref{sksk} we have $|\Re (s_k-s_{k-1})|<|\Im (s_k-s_{k-1})|/8$. Thus by \eqref{sns} we have
$$|\Re s_{N_{m}}\h,^-|<|\Re s_{N_s}\h,^- |+|\Im s_{N_s}|/8<2|x_0|^{-1/2}\ \ \ \ \ \ \ \ \ \ \Box$$

\begin{Corollary}
The solution of the
integral equations (21), (22), with initial condition $(s_{N_s},x_{N_s})$,
exists along $\mathcal{C}$ for $N_m-N_s$ loops. Furthermore, we have
$$0<\Im
 s_{N_{m}}<11/|x_0|,\ \ |\Re s_{N_{m}}+\tfrac43|<2|x_0|^{-1/2},\ \ \ |\frac{x_{n}}{x_0}+i|<(c_7+1)|x_0|^{-1/2}\ln |x_0|$$
for all $N_s< n\le N_m$ for some constant $c_7$.

\end{Corollary}

\section{Asymptotics of the discrete constants. Proof of Theorem\,\ref{recu1} (ii)}\label{PfTh2ii}

\subsection{Asymptotics of the discrete constants of motion}\label{disc}
We derive  two more orders of these formulas, needed in the calculation of $\mu$, cf. Proposition
\ref{sto}. For this we need more properties of functions $F$ and $G$ in Proposition
\ref{fg01}. Denote
\begin{align}
  &F_n=F(s_n,x_n,J_n), \,G_n=G(s_n,x_n), \tilde{F}_n=\tilde{F}(s_n,x_n,J_n),\,\,\tilde{G}_n=\tilde{G}(s_n,x_n)\\
&F_{n;a}=F(s_0+\tfrac{48n}{5x_0},x_0,-\tfrac{24}{5})\text{ and }
G_{n;a}=G(s_0+\tfrac{48n}{5x_0},x_0)\\
&\tilde{F}_{n;a}=\tilde{F}(s_{N_{m}}-\tfrac{48ni}{5x_{N_{m}}},x_{N_{m}},-\tfrac{24i}{5}),\,\,
\tilde{G}_{n;a}=\tilde{G}\(s_{N_{m}}-\tfrac{48ni}{5x_{N_{m}}},x_{N_{m}}\)
\end{align}

Let by convention $B_{0}=\tilde{B}_{0}=0$.

\begin{Lemma}\label{valgab}
 $B_n$ and $\tilde{B}_n$ defined in \eqref{eq:defbn} and  \eqref{eq:defgab} for $n\geq 1$ satisfy
 $$B_n=\frac{48}{5}\sum_{k=0}^{n-1}(F_{k;a}+G_{k;a}),\,\ \label{defbn} \tilde{B}_n=\frac{48i}{5}\sum_{k=0}^{n-1}(\tilde{F}_{k;a}+\tilde{G}_{k;a})$$

 \end{Lemma}

\begin{proof}
 The sums above are, up to elementary sums, telescopic;
 the calculations are straightforward.
\end{proof}

\begin{Note}\label{u04}
For generic $u_0$, $\tilde{B}_n$ would contain a term of order $\ln(n+1)$, but the term vanishes for the special choice $u_0=-4$, which makes the calculation simpler.
\end{Note}
In the following $O(\cdot)$ denotes  $n-$independent error terms.

With $N_0$ defined in \eqref{defs} we study the regions  $0\Le n \Le N_0$ and $N_m-N_0\Le n \Le N_m$. The following estimate is needed.
\begin{Lemma}\label{fngn}
For $n\Le 2N_0$ we have
\begin{equation}\label{ffgg}
|F_n-F_{n;a}|+|G_n-G_{n;a}|+|\tilde{F}_n-\tilde{F}_{{n};a}|+|\tilde{G}_{n}-\tilde{G}_{{n};a}|=O(x_0^{-1}\ln x_0)
\end{equation}
\end{Lemma}
\begin{proof}
Define $\delta_{J;n}$ and $\delta_{z;n}$ by $J_{n}=-48/5(1+\delta_{J;n})$ and
$s_{n}=(s_{0}+\frac{48n}{5x_0})(1+\delta_{z;n})$.  It follows from Lemma \ref{ass} that
\begin{equation}\label{jxs}
\delta_{J;n}=O\(\frac{n+1}{x_0}\);\ x_{n+1}-x_{n}=O(\ln x_0);\
s_{n+1}-s_{n}=
\frac{-2J_0}{x_n}+O\(\frac{(n+1)\ln{x_0}}{x_0^2}\)\\
\end{equation}
implying \vspace{-0.25cm}$$s_{n}=s_{0}+\frac{48n}{5x_0}+O(x_0^{-2}\ln{x_0})$$\vspace{-0.25cm}
 and
 \begin{equation}\label{snxn}
   s_nx_n-s_{n-1}x_{n-1}=\frac{48}{5}+O\(\frac{(n+1)\ln{x_0}}{x_0}\)\Rightarrow
 \delta_{z;n}=O((n+1)x_0^{-1}\ln{x_0})
 \end{equation}
The estimates for $F_n-F_{n;a}$ and $G_n-G_{n;a}$ follow by Taylor expansion,  using \eqref{jxs}
and \eqref{snxn} and the fact that $\frac{N_0}{x_0}=o(1)$.
The proof for $\tilde{F}$ and $\tilde{G}$ is analogous.
\end{proof}

\subsection{Proof of Theorem \ref{recu1} (ii)}\label{partii}
{\em{Case}} {\bf{I.}} Consider $ n \Le N_0$. By \eqref{jxs} we have
\begin{equation}\label{jdif}
(J_n+24/5)(G_n+F_n)=O(\ln x_0/x_0)
\end{equation}

We then have by \eqref{xje1}, Lemma \ref{fngn} and \eqref{jdif}
\begin{equation}\label{xjxj}
\mathcal{Q}_n-\mathcal{Q}_0=
B_n+O(x_0^{-1/4}\ln x_0)
\end{equation}
{\em{Case}} {\bf{II.}} Consider $n$ with $N_0< n \Le N_m-N_0$. We first show that
$|s_n|>\tfrac12|x_n|^{-\frac14}$ and $|s_{n-}|>\tfrac12|x_n|^{-\frac14}$ for
$N_0/2\Le n\Le N_{m}-N_0/2$. We need to distinguish two subcases.

{\bf{II.a}} For $n\Le \eta_1|x_0|/8 := N_2$ (note that $N_2>N_0$),  Lemma \ref{eps1}
 and Lemma \ref{ass} imply
 \begin{equation}\label{sns01}
\left|s_n-s_0-\frac{48ni}{5|x_0|}\right|<1/|x_0|;\ |x_n/x_0-1|<1/20
 \end{equation}
 and similarly
  \begin{equation}\label{sns02}
 \left|s_n -s_{N_{m}}-\frac{48(N_{m}-n)i}{5|x_0|}\right|<1/|x_0|;\ |x_n/x_0+i|<1/20
  \end{equation}
 for $N_{m}-N_2< n\Le N_{m}$.

 {\bf{II.b}} For $N_2< n\Le N_{m}-N_2$ we have $|s_n|\Ge c_1\eta_1/2$ and $|s_{n}\h,^-|\Ge c_1\eta_1/2$ by Lemma \ref{Consistency1}. A
 straightforward calculation using \eqref{sns01} and \eqref{sns02} shows that $|s_n|>\frac12|x_n|^{-\frac14}$ and $|s_{n}\h,^-|>\frac12|x_n|^{-\frac14}$ for
$N_0/2\Le n\Le N_{m}-N_0/2$.

It follows from Lemma \ref{m0}, \eqref{sn2} and Lemma \ref{Consistency1} that
\begin{multline}\frac{1}{R(v,s(v))}-\frac{1}{R_n(v)}=\frac{s_n-s(v)}{R_n(v)R(v,s(v))(R_n(v)+R(v,s(v)))}
\\=O\(x_0^{-1}s_n^{-3/2}\)+O\(x_0^{-1}(s_{n}\h,^-)^{-3/2}\)
\end{multline}
This equation together with \eqref{rsrn0}, \eqref{sn2}, and \eqref{xn2} implies
that
\begin{multline*}
x_{n+1}=x_n+L_n-\oint\frac{s(v)-s_n}{2R_n(v)^3}dv+O\((x_0s_n)^{-2}\)+O\((x_0s_{n}\h,^-)^{-2}\)
=x_n+L_n\\+\oint\frac{J_n(v)}{x_n R_n(v)^3}dv+O\(x_0^{-3/2}\)
=x_n+L_n+\frac{1}{x_n}\oint \tilde{J}_n(v)\frac{\partial Q(v,s_n)}{\partial v}dv+\frac{\rho(s_n)J_n^2}{2x_n}+O\(x_0^{-3/2}\)\\
\end{multline*}
In the equation above we used  \eqref{eq:iden1} to integrate by parts:

$$\oint\frac{\tilde{J}_n(v)}{R_n(v)^3}dv=\oint \tilde{J}_n(v)\frac{\partial Q(v,s_n)}{\partial v}dv+\frac{\rho(s_n)J_n^2}{2}=J_nQ(u_n,s_n)+\frac{\rho(s_n)J_n^2}{2}$$
since $R_n(v)Q(v,s_n)$ is analytic and its loop integral is 0.
Therefore
\begin{equation}\label{xmid}
x_{n+1}=x_n+L_n+\frac{J_nQ(u_n,s_n)}{x_n}+\frac{\rho(s_n)J_n^2}{2x_n}+O\(x_n^{-3/2}\)
\end{equation}
We rewrite \eqref{jn1} using \eqref{sn2}, \eqref{rsrn0} and \eqref{eq:iden1} as
\begin{multline}\label{dJn}
J_{n+1}-J_n=\oint\frac{s_{n+1}-s_n}{R_n(v)+R_{n+1}(v)}dv\\
=-\frac{2J_n}{x_n}\oint\frac{1}{(R_n(v)+R_{n+1}(v))}dv+
\frac{J_nL_n^2}{x_n^2}+O\(x_0^{-3}s_n^{-1}\)+O\(x_0^{-3}(s_{n}\h,^-)^{-1}\)\\
=-\frac{2J_n}{x_n}\(\frac{L_n}{2}-\oint\frac{-J_n(v)}{x_nR_n(v)(R_n(v)+R_{n+1}(v))^2}dv\)+
\frac{J_nL_n^2}{x_n^2}+O\(\frac{\ln s_n}{x_0^3s_n}\)+O\(\frac{\ln s_{n}\h,^-}{x_0^3s_{n}\h,^-}\)
\end{multline}
 Now
\begin{multline}\label{IntRn}
\oint\frac{1}{R_n(v)(R_n(v)+R_{n+1}(v))^2}dv=
\oint\frac{1}{4R_n^3(v)}dv+O\(x_n^{-1}(s_ns_{n}\h,^-)^{-2}\)\\
=\frac{\rho(s_n)J_n}{4}+O\(x_n^{-1}(s_ns_{n}\h,^-)^{-2}\)
\end{multline}
 Using  \eqref{dJn} and \eqref{IntRn} we get
\begin{equation}\label{jmid}
J_{n+1}-J_n=-\frac{J_nL_n}{x_n}+\frac{J_n^2L_n^2}{x_n^2}-\frac{\rho(s_n)J_n^3}{2x_n^2}+O(x_0^{-5/2})
\end{equation}
which, combined with  \eqref{xmid} implies
\begin{equation}\label{xjmid}
x_{n+1}J_{n+1}-x_nJ_n=x_n^{-1}Q(u_0,s_n)J_n^2+O(x_n^{-3/2})=-\tfrac12 Q(u_0,s_n)J_n(s_{n+1}-s_n)+O(x_0^{-3/2})
\end{equation}
On the other hand,
$$\frac{d Q(u_n,s)J(s)}{d
  s}=O\((s_ns_{n}\h,^-)^{-2}\)=O(x_0^{1/2})$$
implying
\begin{equation}
  \label{eq:dQ}
 Q(u_0,s_0)J(s_n)-Q(u_0,s)J(s)=O(x_0^{-1/2})
\end{equation}
for $s$ between $s_n$ and $s_{n+1}$, and thus integrating  \eqref{eq:dQ} we get
\begin{equation}\label{qjmid}
(s_{n+1}-s_n)Q(u_0,s_n)J_n=\int_{s_n}^{s_{n+1}}Q(u_0,s)J(s)ds+O(x_0^{-3/2})
\end{equation}
It follows from \eqref{xjmid},  \eqref{qjmid} and Lemma \ref{ass} that
$$x_{n+1}J_{n+1}-x_nJ_n= -\frac{1}{2}\int_{s_n}^{s_{n+1}}Q(u_0,s)J(s)ds+O(x_0^{-3/2})$$
Summing in $n$ we get
\begin{equation}\label{qn0q0}
\mathcal{Q}_{n}=\mathcal{Q}_{N_0}-\frac{1}{2}\int_{s_{N_0}}^{s_{n}}Q(u_0,s)J(s)ds+O(x_0^{-1/2})
\end{equation}
Now by {\bf{I.}} we have $\mathcal{Q}_{N_0}-\mathcal{Q}_{0}=\frac{4\sqrt{3}i}{5}\ln (N_0)+g_a+O(x_0^{-1/4}\ln x_0)$.
Since $Q(u_0,s)J(s)=-\frac{8\sqrt{3}i}{5s}+O(x_0\ln s)$ by definition and
$s_{N_0}=\frac{48N_0 }{5x_0}+O(x_0^{-1})$ by \eqref{sns01}, we have
$$\int_{s_{0}}^{s_{N_0}}Q(u_0,s)J(s)ds=-\frac{8\sqrt{3}i}{5}\ln \frac{48N_0}{5s_0x_0}+O(x_0^{-1/4}\ln x_0)
$$
Thus by \eqref{qn0q0} we have
$$\mathcal{Q}_{n}=\mathcal{Q}_{0}+g_a+\frac{4\sqrt{3}i}{5}\ln\frac{5s_0x_0}{48}
-\frac{1}{2}\int_{s_{0}}^{s_{n}}Q(u_0,s)J(s)ds+O(x_0^{-1/4}\ln x_0)$$

{\em{Case}} {\bf{III.}} The remaining case  $N_m-N_0< n \Le N_m$ is similar to {\bf{I.}} by symmetry and we omit the details. We get
\begin{equation}\label{xjxjn}
\mathcal{Q}_{N_m}-\mathcal{Q}_n=
\tilde{B}_{N_m-n}+O(x_0^{-1/4}\ln x_0)
\end{equation}
\subsection{Proof of Theorem \ref{recu1} (iii)}\label{partiii}
It follows from \eqref{sn2} that
\begin{multline}\label{sndd}
\frac{s_{n+1}-s_n}{J_n^2}=-\frac{2}{\mathcal{Q}_n}+ \frac{2L_n}{x_n^2J_n}+O\(x_0^{-3}\ln^2 x_0\)
\\
=-\frac{2}{\mathcal{Q}_0}+ \frac{2L_n}{x_n^2J_n}+
\frac{2(\mathcal{Q}_n-\mathcal{Q}_0)}{x_0^2J_0^2}+O\(x_0^{-3}\ln^2 x_0\)
\end{multline}
Now using the definitions of $J$ and $L$ we have for $s(u)$ with $u$ in the $n$th loop on $\mathcal{C}$  (between $u_n$ and $u_{n+1}$)
\begin{multline}\label{j2j2d}
\frac{1}{J_n^2}-\frac{1}{J^2(s)}=\frac{(J_n+J(s))L_n(s-s_n)}{2J_n^2J^2(s)}+O\(\frac{(s-s_n)^2}{s_n}\)\\
=\frac{L_n(s-s_n)}{J_n^3}+O\(\frac{(s-s_n)^2}{s_n}\)+O\(\frac{(s-s_n)^2}{s_{n-}}\)+O\(\frac{\ln x_0(s-s_n)}{x_0}\)
\end{multline}
Integrating both sides gives
\begin{multline}\label{jsquare}
\frac{s_{n+1}-s_n}{J_n^2}-\int_{s_n}^{s_{n+1}}\frac{1}{J^2(s)}ds=\frac{L_n(s_{n+1}-s_n)^2}{2J_n^3}+O\(\frac{1}{s_n x_0^3}\)+O\(\frac{1}{s_{n}\h,^- x_0^3}\)+O\(\frac{\ln x_0}{x_0^3}\)\\
=\frac{2L_n}{x_n^2J_n }+O\(\frac{1}{s_n x_0^3}\)+O\(\frac{1}{s_n\h,^- x_0^3}\)+O\(\frac{\ln x_0}{x_0^3}\)
\end{multline}
This together with \eqref{sndd} implies
\begin{equation}\label{j2q}
\int_{s_n}^{s_{n+1}}\frac{1}{J^2(s)}ds=-\frac{2}{x_0J_0}+\frac{2(\mathcal{Q}_n-\mathcal{Q}_0)}{x_0^2J_0^2}+O\(x_0^{-3}\ln^2x_0\)
\end{equation}
Using \eqref{eq:eqj21} to integrate $1/J^2$ we get
$$\int_{s_n}^{s_{n+1}}\frac{1}{J^2(s)}ds
=-\frac{5}{24\pi i}(\mathcal{K}(s_{n+1})-\mathcal{K}(s_n))$$
This together with \eqref{j2q} implies
$$\mathcal{K}(s_{n+1})-\mathcal{K}(s_n)=\frac{48\pi i}{5\mathcal{Q}_0}+\frac{2\pi i(\mathcal{Q}_n-\mathcal{Q}_0)}{x_0^2J_0}+O\(x_0^{-3}\ln x_0\)\ \ \ \ \text{for } 0\leq n<N_m$$
Summing in $n$ we get
\begin{equation}
  \label{eq:eq44d}
\mathcal{K}(s_n)= \mathcal{K}(s_0)+\frac{48\pi i n}{5\mathcal{Q}_0}+\frac{2\pi i \sum_{j=0}^{n-1}(\mathcal{Q}_j-\mathcal{Q}_0)}{x_0^2J_0}+O\(x_0^{-2}\ln^2 x_0\)
\end{equation}
Now by \eqref{eq:defgab} and Theorem \ref{recu1} (ii) we have $|B_k|\Le \mathfrak{c} \ln x_0$. Thus for $0\Le n\Le N_0$ we have
\begin{equation}\label{q0qn1}
\sum_{j=0}^{n}(\mathcal{Q}_j-\mathcal{Q}_0)=\sum_{j=0}^{n} B_j=O(x_0^{3/4} \ln x_0)
\end{equation}
while for $N_0< n \Le N_m-N_0$ we have
\begin{multline}\label{q0qn2a}
\sum_{j=0}^{n}(\mathcal{Q}_j-\mathcal{Q}_0)=\sum_{j=0}^{N_0}B_j+\sum_{j=N_0+1}^{n}(\mathcal{Q}_j-\mathcal{Q}_{N_0})
+(n-N_0)B_{N_0}\\
=nB_{N_0}-\frac{1}{2}\sum_{j=N_0}^{n}\int_{s_{N_0}}^{s_{j}}Q(u_0,s)J(s)ds+O(x_0^{3/4} \ln x_0)
\end{multline}

Now by Lemma \ref{ass} and Theorem \ref{recu1} (ii) we have
$\frac{(s_{j+1}-s_j)}{2J_j^2}=-\frac{1}{x_0J_0}+O(x_0^{-2}\ln x_0) $. Thus
\begin{multline}\label{q0qn2b}
-\frac{1}{2}\sum_{j=N_0}^{n}\int_{s_{N_0}}^{s_{j}}Q(u_0,s)J(s)ds
=x_0J_0\sum_{j=N_0}^{n}\frac{(s_{j+1}-s_j)}{-2J_j^2}\(-\frac{1}{2}\int_{s_{N_0}}^{s_{j}}Q(u_0,s)J(s)ds\)+O(\ln x_0)\\
=x_0J_0\int_{s_{N_0}}^{s_n}\frac{1}{4J^2(s)}\int_{s_{N_0}}^{s}Q(u_0,s)J(s)ds+O(\ln x_0)
\end{multline}
where we noted that the middle term is a Riemann sum, that we replaced by an
integral plus the usual error bound in terms of the derivative. Using \eqref{eq:eqj21} to write $1/J^2$ in terms of $(\mathcal{K}-1)'$ and
integrating by parts we get
\begin{multline}
  \label{eq:eqct2a}
\int_{s_{N_0}}^{s_n}\frac{1}{4J^2(s)}\int_{s_{N_0}}^{s}Q(u_0,t)J(t)dtds=
-\frac{5}{96\pi i}\int_{s_{N_0}}^{s_n}\(\int_{s_{N_0}}^{s}Q(u_0,t)J(t)dt\)(\mathcal{K}(s)-1)'ds\\
= -\frac{5}{96\pi i}(\mathcal{K}_n-1)\int_{s_{N_0}}^{s_n}Q(u_0,s)J(s)ds -\frac{5}{96\pi i}\int_{s_{N_0}}^{s_n}Q(u_0,s)(J(s)-\hat{J}(s))ds+O(x_0^{-1/4})\\
=-\frac{5}{96\pi i}\int_{s_{N_0}}^{s_n}Q(u_0,s)(\mathcal{K}_nJ(s)-\hat{J}(s))ds+O(x_0^{-1/4})
\end{multline}
Combining \eqref{q0qn2a}, \eqref{q0qn2b} and \eqref{eq:eqct2a} we have
\begin{equation}\label{q0qn2}
\sum_{j=0}^{n}(\mathcal{Q}_j-\mathcal{Q}_0)=nB_{N_0}-\frac{5\mathcal{Q}_0}{96\pi i}\int_{s_{N_0}}^{s_n}Q(u_0,s)(\mathcal{K}_nJ(s)-\hat{J}(s))ds+O(x_0^{3/4}\ln x_0)
\end{equation}

Since Theorem \ref{recu1} (ii) implies $\mathcal{Q}_n-\mathcal{Q}_{N_m}=O(1)$ for $N_m-N_0\Le  n\Le N_m$, we see that
\eqref{q0qn2} is also valid for $N_m-N_0\Le  n\Le N_m$.

Since $K_n$ is bounded, \eqref{eq:eq44d} and \eqref{q0qn2} imply that
\begin{equation}\label{knk0}
\mathcal{K}_n=-\frac{2\pi i n}{x_0}+ O(\ln x_0/x_0)
\end{equation}
Note also $\mathcal{Q}_0=-24x_0/5+O(\ln x_0)$ since $J_0=-24/5+O(x_0^{-1}\ln x_0)$. This together with \eqref{q0qn2} and \eqref{knk0} imply that
\begin{equation}\label{q0qn3}
\sum_{j=0}^{n}(\mathcal{Q}_j-\mathcal{Q}_0)=nB_{N_0}+\frac{1}{4\pi i}\int_{s_{N_0}}^{s_n}Q(u_0,s)(-2\pi i nJ(s)-x_0\hat{J}(s))ds+O(x_0^{3/4}\ln x_0)
\end{equation}
Now,  by definition,  $Q(u_0,s)(-2\pi i nJ(s)-x_0\hat{J}(s))=-\frac{16\sqrt{3}\pi}{5s}n+O(x_0\ln s)$ and by \eqref{sns01} we have
$s_{N_0}=\frac{48N_0 }{5x_0}+O(x_0^{-1})$.
Thus
\begin{multline}
\int_{s_{N_0}}^{s_n}Q(u_0,s)(-2\pi i nJ(s)-x_0\hat{J}(s))ds=\\
\int_{s_{0}}^{s_n}Q(u_0,s)(-2\pi i nJ(s)-x_0\hat{J}(s))ds+\frac{16\sqrt{3}\pi}{5}n\ln \frac{48N_0}{5s_0x_0}+O(x_0^{3/4}\ln x_0)\end{multline}
This together with \eqref{eq:defgab} and \eqref{q0qn3} implies
\begin{equation}\label{q0qnf}
\sum_{j=0}^{n}(\mathcal{Q}_j-\mathcal{Q}_0)= ng_a+\frac{1}{4\pi i}\int_{s_{0}}^{s_n}Q(u_0,s)(-2\pi i nJ(s)-x_0\hat{J}(s))ds+\frac{4\sqrt{3} i}{5}n\ln\frac{5s_0x_0}{48}+O(x_0^{3/4}\ln x_0)
\end{equation}
Comparing this with \eqref{q0qn1} we see that \eqref{q0qnf} is in fact valid for $0\le n\le N_m$. The conclusion then follows from \eqref{eq:eq44d} and \eqref{q0qnf}.

\subsection{Proof of Proposition\,\ref{c123}}\label{Pfc123}

(i) It follows from Lemma \ref{eps1} and \eqref{eq:eq44d} that
$$ 1+\frac{\pi(s_n+\tfrac43)}{J(-\tfrac43)}= \frac{\pi i s_0}{J(0)}-\frac{2\pi i J(0)n}{J_{0}x_{0}}+\frac{2\pi i \phi_n}{x_0^2J(0)}+O(x_0^{-5/4}\ln x_0)+O\((s_{n}\h,^-)^{3/2}\)$$
This implies \eqref{snn}. A calculation using \eqref{eq:defgab} and Theorem
\ref{recu1} (iii) shows that $\Re \frac{2\phi_n}{x_0^2}=O(x_0^{-1})$ and $\Im \frac{2\phi_n}{x_0^2}=O(x_0^{-1}\ln x_0)$. Since
$0<\Im s_{N_m}<11/|x_0|$, \eqref{snn} implies $N_{m}=\frac{|x_0|}{2\pi}+O(\ln x_0)$.

(ii) This follows directly from (i).

(iii) \eqref{xnm1} follows from Theorem \ref{recu1} (ii). By (i) and (ii) we have $s_{N_{m}\h,^-} =O(x_0^{-1}\ln x_0)$, and thus by Lemma \ref{eps1} we have $J_{N_{m}}=iJ_0+O(x_0^{-1}(\ln x_0)^2)$. The rest follows from \eqref{xnm1}.
$\Box$

\section{Application: finding the Stokes multiplier}\label{find}
As an application of the discrete constants of motion, in this section we  find the
Stokes multiplier $\mu$ by
analyzing the tritronqu\'ee solution $y_t(z)$ of P$_1$ specified by the sector of analyticity
\eqref{sectorinz}.

\subsection{Overview of the approach}\label{Overmu}

 The solution $y_t$ is meromorphic; this
was known since Painlev\'e, and proving meromorphicity does not require a
Riemann-Hilbert reformulation, see e.g. \cite{CostinOR,Hinkkanen} for direct
proofs and references.
  Starting with  a large $z\in\RR^+$ we analytically continue $y_t$
(i) anticlockwise on an arccircle until $\arg z=\pi$   and (ii) clockwise on
an arccircle until $\arg z=-\pi$. The continuation (ii) traverses the pole
sector, $\arg z\in (-\pi,-3\pi/5)$.
Because of the above-mentioned
meromorphicity, we must have
\begin{equation}
  \label{eq:mer1}
  y_t(|z|e^{i\pi})=y_t(|z|e^{-i\pi})
\end{equation}

After the normalization \eqref{chvarx} this tritonque\'e $y_t(z)$ becomes $h_t(x)$, solution of \eqref{eq:eqp} specified by  \eqref{sectorinx}. The analytic continuation corresponds in the new variables to the following: We start with large
$x$ with  $\arg x=\pi/4$ and (i') analytically continue $h_t(x)$ anticlockwise, until $\arg
x=3\pi/2$, and  (ii') analytically continue $h_t(x)$ clockwise, until $\arg x=-\pi$. The
single-valuedness of the solutions of equation \eqref{eq:mer1} implies
 \begin{equation}\label{hpm}
h_t(|x|e^{3\pi i/2}) = -h_t(|x|e^{-\pi i})-2+\frac{8}{25|x|^2}
\end{equation}

Recall that a {\em{Stokes line}} is a direction at which the constant  $C$ in the transseries of solutions changes: the {\em{Stokes phenomenon}}, and in fact $C=C(\arg x)$ is piecewise constant, see \cite{cch}; orthogonal to them are the {\em{antistokes lines}}, directions along which some exponential in the transseries solutions is purely oscillatory. See \cite{imrn,cch}.
By Theorem 2  (iii) of \cite{imrn}
the value of $C$ jumps by $\mu$, cf. also \cite{cch}.

By Theorem 2 of \cite{imrn}, $\RR^+$ and $\RR^-$ are the (only) Stokes lines of \eqref{eq:eqp} (the Stokes lines coincide with directions along which some exponenetial in the transseries has maximal decay)
and the antistokes lines are $i\RR^+$ and $i\RR^-$. The tritronque\'e $h_t$, with zero constant in its transseries in the first quadrant, $C(\arg x)=C_+=0$ for $\arg x\in(0,\frac{\pi}{2})$, is analytically  continued (i') traversing the antistokes line $\arg(x)=\frac{\pi}{2}$ ($C$ does not change) and reaches the Stokes line $\arg x=\pi$, where $C_-=\mu$; $h_t$ continues to have a transseries with the same $C$ until the next antistokes line $\arg x=\frac{3\pi}{2}$ beyond which it enters a pole region; upon analytic continuation  (ii')  $h_t$ traverses the Stokes line $\arg x=0$ gives $C(0-)=-\mu$, then crosses the antistokes line $\arg(x)=-\frac{\pi}{2}$ entering the pole sector.

For $y_t(z)$ continuation (i)  means that $z$ traverses the antistokes line $\arg(z)=\pi/5$ and reaches the Stokes line $\arg z=3\pi/5$, while (ii') traverses the Stokes line $\arg z=-\pi/5$, the antistokes line $\arg(z)=-3\pi/5$, entering the pole sector.

In variable $z$, and using the five-fold symmetry, we see that

\begin{Note}\label{StokesAz} Their position in the original $z$ plane are $\arg z\in\{-\pi/5,3\pi/5,
7\pi/5\}$ (Stokes)
and $\arg z\in\{-3\pi/5,\pi/5, \pi\}$ (antistokes).
The lines bordering the sectors of symmetry \eqref{eq:sectorsz} are antistokes lines for some tritronque\'e.
\end{Note}

Going back to the normalized form $h_t(x)$, along
$\RR^+$ the change is given by
\begin{equation}
  \label{eq:sjump}
C_-=:C(0^-)=C_++\mu=:C(0^+)+\mu
\end{equation}
See also \cite{cch}, where we also link \eqref{eq:sjump} to the singularities in Borel plane.   For
the tritronqu\'ee $C(0_+)=0$. Along $\RR^-$, we have
$C(\pi+0)=C(\pi-0)-\mu=C(0^+)-\mu =-\mu$ for the same
$\mu$ as in \eqref{eq:sjump} because of Lemma 17 in \cite{cch} and since the direction of
continuation in (i) is opposite to that in (ii). In (ii), the third quadrant,
a sector with poles in $x$ is traversed. In this region $h_t$ is described by constants of motion (cf. Theorem \ref{recu1} (ii) and \eqref{snn}), which are
valid until $x$ reaches $\mathbb{R}^-$ when it is {\em again described by a transseries}; the asymptotic expansions of the
constants of motion that we obtain depend on $C$. The transseries
representation of $h_t$ also depends on $C$ in a way visible in the first few
terms when $\arg x=-\pi$ or $3\pi/2$. Eq. \eqref{hpm} is a nontrivial
equation for $\mu$ which determines it uniquely. The fact that $\mu$ is uniquely
determined is not surprising given that there is only one solution, the
tritronqu\'ee, with
algebraic behavior in the region \eqref{sectorinz}, cf. \cite{invent}, Proposition
15.
\subsection{The transseries regions}
Our goal is to find the value of the Stokes multiplier $\mu$ using \eqref{hpm}. By \cite{invent} $h$ has the asymptotic expansion \eqref{eq:eq51} in the region $\Im x<0,~~\Re x \in[-\tfrac43\ln |x|,0]$. Similarly, since $y(z)$ is continuous in $z$, by \eqref{chvarx} we have when $h_t(x)\sim  -2$ when $x$ is in the region $\Re x<0,~\Im(x)\in[-\tfrac43\ln |x|,0]$. A calculation similar to \eqref{eq:eq51} (cf.\cite{invent}) gives the asymptotic expansion
\begin{equation}
  \label{eq:eq51n}
 h_t(x)\sim  -2-h_{0}(\tilde{\xi})-\frac{1}{x}h_{1}(\tilde{\xi})-\frac{1}{x^2}h_{2}(\tilde{\xi})+\cdots
\end{equation}
where $\tilde{\xi}=\tilde{\mu} e^{ix}$ and  $\tilde{\mu}=-\mu$ (see the
discussion below \eqref{eq:sjump}).

\begin{Remark}
{\rm The fact  that $\tilde{\mu}=-\mu$  is in fact not used for our purpose of calculating  $\mu$}.
\end{Remark}

\begin{Note} There are infinitely many points $x_0$ so that $h_t(x_0)=-4$, and among them there are sequences with modulus going to $\infty$.
\end{Note}

\begin{Proposition}\label{x0j0tt}
One can choose $x_0$ satisfying Assumption \eqref{ASSS1} with $|x_0|$ is sufficiently large, such that the tritronquee solution with $u_0=h_t(x_0)=-4$ satisfies
\begin{equation}\label{s0}
s_0=\frac{8(3+\sqrt{3}i)}{5x_0}+O(x_0^{-3/2})
\end{equation}
and
\begin{multline}\label{jx0}
\frac52 x_0J_0=24 i \pi  k_0-2i \sqrt{3} \ln k_0+
12 \ln \left[(1+i) \left(\sqrt{3}+i\right)\mu^{-1}\right]\\
-2 \left(3^{-\frac12}-i\right) (5 \pi +3 i)-\sqrt{3} i \Big(6\ln 2+3\ln
3+2\ln 5\Big)-3 \ln \tfrac{100}{3}-2 \sqrt{3} i \ln \pi
+O\(\frac{\ln k_0}{k_0}\)
\end{multline}
\end{Proposition}

\begin{proof}
Since $u_0=-4$, \eqref{eq:eq51} implies that ${\xi}(\xi/12-1)^{-2}=-4$ for
$x$ near $-i\mathbb{R}$. This equation has solutions
$\xi=6(-1\pm\sqrt{3}i)$. For convenience we choose $\xi=6(-1+\sqrt{3}i)$. Let
$x_0$ be a value of $x$ corresponding to $\xi$. A straightforward calculation using \eqref{eq:eq51} shows \eqref{s0}.

We write $x_0=-2k_0\pi i+\tilde{x}_0$ where $k_0\in \mathbb{N}$ is large, and
$\tilde{x}_0=O(\ln k_0)$. By definition of $\xi$ we see that $\tilde{x}_0$ solves the equation $$\frac{\mu
  e^{-\tilde{x}_0+\pi i/4}}{\sqrt{2k_0\pi +i\tilde{x}_0}}=6(-1+\sqrt{3}i)$$ Expanding the square root at $\tilde{x}_0=0$ and inverting the
exponential we obtain
\begin{equation}\label{x0}
\tilde{x}_0=-\ln\(6(-1+\sqrt{3}i)\mu^{-1}\sqrt{-2k_0\pi i}\)+O(\frac{1}{k_0})
\end{equation}

Combining \eqref{x0} and Proposition \ref{js0} we obtain \eqref{jx0}.

\end{proof}

Now we have
\begin{Proposition}\label{reach} Let $x_0$ as in {Proposition}\,\ref{x0j0tt}, large enough so that
$(s_0,x_0)\in\mathcal{R}_1$, so that $(s_n,x_n)$ exist for
  $0\Le n\Le N_{m}$. Furthermore, $x_{N_{m}}$ is in the transseries region
  $\{x\in \mathbb{C}:\Re x<0,\Im x \in(-\tfrac43\log |x|,0)\}$ and
  \eqref{eq:eq51n} implies
\begin{equation}\label{snn2}
s_{N_m}=\frac{24}{5x_{0}}+O(x_{0}^{-2}\ln x_{0} )
\end{equation}
\end{Proposition}

\begin{proof}
 With $u_0=-4$ and $u_0,x_0,s_0$ given by Proposition \ref{x0j0tt}, the conditions of Proposition \ref{Wholeinterval} are satisfied. By Lemma \ref{eps1}, Proposition \ref{c123} (iii) and \eqref{jx0} we have
$$\Im x_{N_{m}}=\Im \frac{x_0J_0}{J_{N_{m}}}+O(1)=\Im
\frac{5\, |x_0|\,s_{N_{m}}\h,^-\, \ln s_{N_{m}}\h,^-}{48}+O(1)$$ Since $0<\Im s_{N_{m}}<11/|x_0|$ by Proposition \ref{Wholeinterval}, and $\Re s_{N_{m}}\h,^-=O(1/x_0)$ by \eqref{sn8} and \eqref{jx0}, we see that
$\Im x_{N_{m}}>-\frac{55}{48}\ln|x_0|+O(1)$. Since the second quadrant is a transseries region (cf. \eqref{sectorinx} and \cite{cch}) and $u_0=-4$, by \eqref{eq:eq51n} we must have $\Im x_{N_{m}}<0$.

It follows from \eqref{eq:eq51n} that and
$-{\tilde{\xi}}(\tilde{\xi}/12-1)^{-2}=-2$ for $x$ near $-\mathbb{R}$, with
solutions $\tilde{\xi}=12 \left(-4\pm\sqrt{15}\right)$, which implies \eqref{snn2} by \eqref{eq:eq51n}. Note that $x_{N_{m}}=-i x_0+O(\ln x_0)$ by Proposition \ref{c123} (iii).

\end{proof}

\subsection{Calculating the Stokes multiplier}
We now find the exact value of the Stokes multiplier using Proposition
\ref{c123} (i) and \eqref{snn2}.
\begin{Note}\label{method}{\rm
 Eq. \eqref{snn} gives a formula for $s_{N_m}$ based on the constants of motion given in Theorem \ref{recu1} (ii) and (iii), whereas \eqref{snn2} gives the value of $s_{N_m}$ according to the asymptotic expansion  \eqref{eq:eq51n} for the tritronqu\'ee. Thus by setting them equal to each other we establish an equation for $\mu$, see \eqref{stok2} below.}
\end{Note}
We need to prove some estimates first.
\begin{Lemma}\label{xjed}
Let $x_0,s_0$ as in {Proposition}\,\ref{x0j0tt}. For $N_0/2<N<2N_0$ we have
\begin{multline}\label{xjee1}
B_{N}=\frac{2}{5} \bigg(\left(- i+\frac{1}{\sqrt{3}}\right) \pi \\+6+2 i \sqrt{3}+i \sqrt{3} \ln 3 +\ln 27 -6\ln(2 \pi )\bigg)+\frac{4\sqrt{3}i}{5}\ln N+O(x_0^{-1/4}\ln x_0)
\end{multline}
where $B_n$ is as defined in \eqref{defbn}.
Similarly for $N_0/2<m<2N_0$ we have
\begin{equation}\label{xjee2}
\tilde{B}_m=\frac{12}{5} (1-\ln \pi )+O(\frac{1}{m+1})+O(x_0^{-1/4}\ln x_0)
\end{equation}
Equivalently, $g_a=\frac{2}{5} \bigg(\left(- i+\frac{1}{\sqrt{3}}\right) \pi +6+2 i \sqrt{3}+i \sqrt{3} \ln 3 +\ln 27 -6\ln(2 \pi )\bigg)$ and $g_b=\frac{12}{5} (1-\ln \pi )$ ($g_a,g_b$ are as defined in \eqref{eq:defgab}).
\end{Lemma}

\begin{proof}
Since $u_0=-4$ and $\xi=6(-1+\sqrt{3}i)$, we have
$s_0=\frac{8}{5x_0}(3+\sqrt{3}i)+O(\ln x_0/x_0^2)$ by direct calculation. Also recall that $J_{00}=-\frac{12}{5}+\frac{4\sqrt{3}i}{5}$ by \eqref{eq:eqj00}.

With this choice we have the following explicit formulas by direct calculation using the definitions of $F$ and $G$ (cf. Proposition \ref{fg01}):
$$F_{n;a}=l_n+O(\frac{\ln x_0}{x_0})$$
where
\begin{multline}\label{}
l_n=-\frac{3}{4}+\frac{-3-i \sqrt{3}-6 n}{12} \ln\left(3+i \sqrt{3}+6 n\right)
+\frac{9+i \sqrt{3}+6 n}{12}\ln\left(9+i \sqrt{3}+6 n\right)
\end{multline}
and
$$G_{n;a}=g_n+O(\frac{\ln x_0}{x_0})$$
where
$$g_n=\frac{1}{2} \ln \left(\frac{6 n+3+i \sqrt{3}}{6 n+6}\right)$$
Thus, by Theorem \ref{recu1} (ii) we have for $N<2N_0$
$$B_{N}=\frac{48}{5}\sum_{k=0}^{N-1}(l_n+g_n)+O(x_0^{-1/4}\ln x_0)$$
The sum is a telescopic sum plus an explicit sum, and we get
\begin{equation}
  \label{eq:eqsum2}
\sum_{k=0}^{N-1}(l_n+g_n)=\frac{1}{12}(\left(3+i \sqrt{3}+6 N\right)
\ln\left(3+i \sqrt{3}+6 N\right)-6N-6\ln6 N-6\ln(N!))
\end{equation}
Using Stirling's formula  $\ln(n!)=(-1+\ln n) n+\frac{1}{2}
\left(-\ln\left(\frac{1}{n}\right)+\ln(2 \pi )\right)+O(1/n)$ in
\eqref{eq:eqsum2} we get
\begin{multline}
\sum_{k=0}^{N-1}(l_n+g_n)=\frac{1}{24} \left(6+2 i \sqrt{3}+i \sqrt{3} \ln 3 +\ln 27 -6\ln(2 \pi )\right)\\+\frac{(-3 i+\sqrt{3}) \pi}{72} +\frac{i}{4\sqrt{3}}\ln N+O(1/N)
\end{multline}
This shows \eqref{xjee1}.

The proof for \eqref{xjee2} is similar. Straightforward calculations using \eqref{snn2}
show that
$$\tilde{B}_{m}=\frac{48i}{5}\sum_{n=1}^{m-1}(\olhatl_n+\olhatg_n)+O(x_0^{-1/4}\ln x_0)$$
where
$$\olhatl_n=-\frac{i}{4} \Big((2 n+1) \ln (2 n-1)-(2 n+1) \ln (2
n+1)+2\Big);\ \olhatg_n=\frac{i}{2} \left(\ln \left(\frac{i}{2}-i n\right)-\ln n+\frac{i \pi }{2}\right)$$
These can be summed in $n$ explicitly implying \eqref{xjee2}.
\end{proof}

\begin{Lemma}\label{qjqj}
For $t\ne 0$ we have
\begin{multline}\label{qj}
\int_{t}^{-\tfrac43}Q(-4,s)J(s)ds=-\frac{8}{5} \bigg(6 \sqrt{3} i \ln 2+\pi  \left(\sqrt{3}-4 i\right)\\
+6 \ln \left(4-\sqrt{15}\right)+2 \sqrt{3} i \ln 3\bigg)+\frac{8\sqrt{3}i}{5}\ln t+O(t\ln (|t|+1))
\end{multline}
and
\begin{equation}\label{qjh}
\int_{0}^{-\tfrac43}Q(-4,s)\hat{J}(s)ds=-\frac{16\sqrt{3} \pi}{5} +\frac{16}{15}\ln \left(4-\sqrt{15}\right)+\frac{32}{3} \ln \left(4+\sqrt{15}\right)
\end{equation}

\end{Lemma}

\begin{proof} The proofs of \eqref{qj} and \eqref{qjh} are very similar.
We have by definition and Lemma \ref{eps1}
\begin{multline}\label{qj2}
\int_{t}^{-\tfrac43}Q(-4,s)J(s)ds=\int_{t}^{-\tfrac43}\frac{-4J(s)}{s \sqrt{-48+9 s}}ds\\
=\int_{0}^{-\tfrac43}\(\frac{-4J(s)}{s \sqrt{-48+9 s}}-\frac{i J(0)}{\sqrt{3}s}\)ds+\frac{i J(0)\ln s}{\sqrt{3}}\big|_{t}^{-\tfrac43}+O(t\ln (|t|+1))\\
=\lim_{\epsilon\to 0}\int_{\mathcal{C}_0}\int_{t}^{-\tfrac43}\(\frac{-4\sqrt{U(\epsilon)+s}}{s \sqrt{-48+9 s}}-\frac{i \sqrt{U(\epsilon)}}{\sqrt{3}s}\)ds\,du+\frac{i J(0)\ln s}{\sqrt{3}}\big|_{t}^{-\tfrac43}+O(t\ln (|t|+1))
\end{multline}
where $U(\epsilon)=u^3/3+u^2+\epsilon i$, and $\mathcal{C}_0$ is as in Corollary \ref{lem2}. In particular it surrounds $-3-\epsilon i/3$, $-2-\sqrt{\epsilon i}$, $\sqrt{-\epsilon i}$, and $1-\epsilon i/3$ but neither $-\sqrt{-\epsilon i}$ nor $-2+\sqrt{\epsilon i}$.

Similarly
\begin{equation}\label{qjh2}
\int_{0}^{-\tfrac43}Q(-4,s)\hat{J}(s)ds=\lim_{\epsilon\to 0}\int_{\hat{\mathcal{C}}_0}\int_{0}^{-\tfrac43}\(\frac{-4\sqrt{U(\epsilon)+s}}{s \sqrt{-48+9 s}}-\frac{i \sqrt{U(\epsilon)}}{\sqrt{3}s}\)ds\,du
\end{equation}
where $\hat{\mathcal{C}}_0$ is as in Corollary \ref{lem2}. In particular it surrounds $-2+\sqrt{\epsilon i}$, $\pm\sqrt{-\epsilon i}$, and $1$, but neither $-3$ nor $-2-\sqrt{\epsilon i}$.

Elementary integration gives
\begin{multline}\label{uu}
\int_{0}^{-\tfrac43}\(\frac{-4\sqrt{U+s}}{s \sqrt{-48+9 s}}-\frac{i \sqrt{U}}{\sqrt{3}s}\)ds=\frac{1}{3} \bigg(\sqrt{3} \pi  \sqrt{U}+4 \ln\left(16-8 i \sqrt{3} \sqrt{U}-3 U\right)\\-4 \ln\left(24-3 U-4 i \sqrt{5} \sqrt{-4+3 U}\right)
+i \sqrt{3} \sqrt{U} \bigg(\ln 48 +\ln U \\-2 \ln(16+3 U)+\ln\left(16-27 U+4 \sqrt{15} \sqrt{U (-4+3 U)}\right)\bigg)\bigg)
\end{multline}
This function  can be integrated in $u$ explicitly as well; the calculation
is tedious but straightforward and we omit the details. The branches of
$\ln$ and square roots are chosen according to analytic continuations along the contour $\mathcal{C}_0$ or
$\hat{\mathcal{C}}_0$ where the initial branch is consistent with $J$ or $\hat{J}$. Integrating \eqref{uu} along $\mathcal{C}_0$ we obtain
\begin{multline}\label{final}
\int_{0}^{-\tfrac43}\(\frac{-4J(s)}{s \sqrt{-48+9 s}}-\frac{i J(0)}{\sqrt{3}s}\)ds\\
=-\frac{8}{5} \left(4 i \sqrt{3} \ln 2+2 \left(-2 i+\sqrt{3}\right) \pi + 6 \ln\left(4-\sqrt{15}\right)+3 i \sqrt{3} \ln 3\right)
\end{multline}
which together with \eqref{qj2} implies \eqref{qj}. Similarly integrating \eqref{uu} along $\hat{\mathcal{C}}_0$ and using \eqref{qjh2} we obtain \eqref{qjh}.
\end{proof}

Now we calculate the Stokes multiplier $\mu$ rigorously using Proposition \ref{c123} (i).

\begin{proof}[Proof of Proposition \ref{sto}]
We apply Proposition \ref{c123} (i) by first noting that
\begin{equation}\label{com22}
\phi_{N_m}=\frac{g_a}{2\pi}+\frac{1}{4\pi i}\int_{s_{0}}^{-\tfrac43}Q(u_0,s)(J(s)-\hat{J}(s))ds+O(x_0^{3/4} \ln x_0)
-\frac{2\sqrt{3}}{5\pi}\ln\frac{3+\sqrt{3}i}{6}
\end{equation}
since $N_m=\frac{|x_0|}{2\pi}+O(\ln x_0)$ by Proposition \ref{c123} and $s_0x_0=\frac{8(3+\sqrt{3}i)}{5}+O(x_0^{-1/2})$ by \eqref{s0}.

Now it follows from Lemma \ref{qjqj}
\begin{multline}\label{eq:eqct2}
 \frac{1}{4\pi i}\int_{s_{0}}^{-\tfrac43}Q(u_0,s)(J(s)-\hat{J}(s))ds=
  \frac{1}{4\pi i}\int_{s_{0}}^{-\tfrac43}Q(u_0,s)J(s)ds\\-
  \frac{1}{4\pi i}\int_{0}^{-\tfrac43}Q(u_0,s)\hat{J}(s)ds+O(x_0^{-1}\ln x_0)
 =\frac{2}{5} \left(\frac{\sqrt{3} \log \left(s_0/576\right)}{\pi }-i \sqrt{3}+4\right) +O(x_0^{-1}\ln x_0)\\
 =-\frac{2 \sqrt{3} \log |x_0|}{5 \pi }+\frac{8}{5}-\frac{2 i}{5 \sqrt{3}}-\frac{\sqrt{3} \log (10800)}{5 \pi }+O(x_0^{-1})
\end{multline}
Applying \eqref{xjee1} and \eqref{eq:eqct2} to \eqref{com22} we obtain
\begin{multline}\label{dsum}
\phi_{N_m}=-\frac{i}{5 \pi } \bigg(\pi  \left(\sqrt{3}+8 i\right)-2 i \sqrt{3}-6-6 i \sqrt{3} \ln 2-4 i \sqrt{3} \ln 3-2 i \sqrt{3} \ln 5\\
-2 \ln 27+2 \sqrt{3} i \ln \left(3+i \sqrt{3}\right)+6 \ln \left(3+i \sqrt{3}\right)+6 \ln (\pi )-2 i \sqrt{3} \ln|x_0|\bigg)+O(x_0^{3/4}\ln x_0)
\end{multline}
By \eqref{snn}, \eqref{jx0}, \eqref{dsum} and   Proposition \ref{reach},
the matching equation \eqref{hpm} implies
\begin{multline}\label{stok2}
-\frac{4 \left(\sqrt{3}-6 i\right)}{5 \pi}=\frac{24 i }{5 \pi}(k_1-N_m)+\frac{1}{5 \pi ^2}\bigg(12 \ln \left(\frac{(6+6 i) \left(\sqrt{3}+i\right)}{\mu}\right)+i \pi -4 \sqrt{3} \pi -24 \ln 2\\
-\ln 729-6 \ln (5 \pi )\bigg)+O(x_0^{-1/4}\ln
x_0) \Rightarrow \mu=e^{2(k_1-N_m)\pi i}\sqrt{\frac{6}{5 \pi }}i=\sqrt{\frac{6}{5 \pi }}i
\end{multline}
since $k_1-N_m\in \mathbb{Z}$.
\end{proof}
\subsection{The Painlev\'e equation P$_2$}\label{secp2} The normal form of
$P_2$ is (\cite{invent})
\begin{align}
  \label{eq:p2n0}
  h''+\frac{h'}{t}-\left(1+\frac{24\alpha^2+1}{9t^2}\right)h-\frac{8}{9}h^3+
\frac{8\alpha}{3t}h^2+\frac{8(\alpha^3-\alpha)}{9t^3}=0
\end{align}
The associated asymptotic Hamiltonian equation, with Hamiltonian $s$  is
$s''-s-\frac89s^3=0$. With $R=\sqrt{9u^2+4u^4+18s(u)}$,  we have (cf. \cite{cch} (37) and (38))
\begin{align}\label{syst22}
  &\frac{ds}{du}=-\frac{8\alpha u^2+R}{3x}+\frac{u(1+24\alpha^2)}{9x^2}+\frac{8(\alpha-\alpha^3)}{9x^3}\\
&\frac{dx}{du}=\frac{3}{R}\label{syst32}
\end{align}
We integrating in $u$  along cycle $\mathcal{C}$ surrounding two or three
singularities, and use the notation \eqref{eq:eqJLu}. Since  $u^2$ is
single-valued  we get
\begin{equation}
  \label{eq:p2ev}
  s_{n+1}-s_n=-\frac{J_n}{3x_n}+O(x_n^{-2});\ \ x_{n+1}=x_n+3L_n
\end{equation}
the same as the case for P$_1$ except for the fact that $R^2$ is now
quartic. The leading order constants of motion are of the same form as those
for P$_1$. We leave this analysis for a different paper.

\section{Appendix}\label{appendix}
\begin{proof}[Proof of Lemma\,\ref{lem1}]
  (i) Analyticity of the roots of a polynomial in $\CC\setminus S_1$ where
  $S_1$ is the finite set of points where the roots coallesce  is standard
  \cite{Ahlfors}; here $S_1=
  \{0,-4/3\}$.  As for the behavior near $S_1$, because of the symmetry \eqref{sym_T}, it suffices to  analyze the roots near $0$.

  We write $\epsilon^2=-s$ and rewrite the equation as $v\sqrt{1+v/3}=\sigma
  \epsilon$, $\sigma=\pm 1$. By symmetry, it is enough to analyze the
  case $\sigma=1$, $v\sqrt{1+v/3}=\epsilon$. We choose the branch of the
  square root with the cut $(-\infty,-3]$. The implicit function theorem (IFT)
  applies at $(v,\epsilon)=(0,0)$ and gives a root, ${\ro}_2(\epsilon)$ which is analytic on the universal covering of $\CC\setminus
  S_1$. Consider the domain $S_2:=\{\epsilon:|\epsilon|<2/3\}$. If $|{\ro}_2|=1$
  have $|{\ro}_2^2||1+{\ro}_2/3|> 1-1/3=2/3$ and, by analyticity and the fact that
  ${\ro}_2(0)=0$ we see that $|{\ro}_2|<1$ for $\epsilon\in S_2$. By our choice of branch, we
  thus have throughout $S_2$,
  \begin{equation}
    \label{eq:bd1}
    |{\ro}_2|<1;\ \ \text{and} \ \ \Re\sqrt{1+r/3}>0
  \end{equation}
Using \eqref{eq:bd1}, we see that
\begin{equation}
  \label{eq:estr2}
  |{\ro}_2-\epsilon|=\left|\frac{\epsilon
      {\ro}_2}{1+\frac{{\ro}_2}{3}+\sqrt{1+\frac{{\ro}_2}{3}}}\right|\Le \frac{|\epsilon|}{2/3+\sqrt{2/3}}\Le\frac{\epsilon}{4}
\end{equation}
Estimating the right side of the equality in \eqref{eq:estr2}, now relying on
$|{\ro}_2-\epsilon|\Le\frac{1}{4}\epsilon$ we get that $|{\ro}_2-\epsilon|\Le |\epsilon^2|/3$.
The result about ${\ro}_1$ follows similarly.

Using the symmetry \eqref{sym_T}, for $|\sm|<\sqrt{2/3}$, {\em in some
  labeling}, the three roots $\tilde{r}_i$ of $P$ satisfy
\begin{equation}
  \label{eq:eqrt}
 |\rho_{1}-\sqrt{\sm}|<|\sm|;\
 |\rho_3-3+\sm/3|<|\sm^{\!\!\!\!2}|;  |\rho_{2}-\sqrt{\sm}|<|\sm|;\ \end{equation}
where $\rho_i=2+\tilde{r}_i$, $\sm=4/3+s$.

(ii) We let $s$ traverse a region  in $\mathbb{H}$.

We note that the roots  do not cross
$\RR$, otherwise we would have $s=-r_j^2-r_j^3/3\in\RR$. As a consequence  of
this and by analyticity $\Im r_j, j=1,2,3$ do not change sign.
  For small $s$,
${\ro}_2\in \HH$, thus ${\ro}_2\in\HH$ for all $s\in\HH$.

In \eqref{eq:eqrt},
$\tilde{\ro}_2$ is the only root in $\HH$, thus $\tilde{\ro}_2={\ro}_2$. Similarly,
since $r_{{\Tco}}\in -\HH$ for small $s$, $r_{{\Tco}}\in -\HH$ for all $s$.

Letting $u=-ti$ where $t\in \mathbb{R}^+$ we see that $u^3/3+u^2+s=s+\frac{i
  t^3}{3}-t^2\ne 0$ since its imaginary part is positive. Similarly letting
$u=-2-ti$ we have $u^3/3+u^2+s=s+\frac{i t^3}{3}+t^2+\frac{4}{3}\ne 0$. Thus
neither ${\ro}_3(s)$ nor ${\ro}_1(s)$ crosses the line $\Re z=0$ or $\Re z=-2$.  This together with
\eqref{eq:eqr} shows that $\Re {\ro}_3(s)>0$ and $\Re {\ro}_1(s)<-2$ for all
$s\in\HH$. Comparing \eqref{eq:eqr} we see that $\tilde{\ro}_3={\ro}_3$ and
$\tilde{\ro}_1={\ro}_1$. Finally, for  small $s\in\HH$, by \eqref{eq:eqr},
 ${\ro}_2$ is between
 $l_1=\{t(1+i): t\Ge 0\}$ and $l_2=\{-2+t(-1+i): t\Ge
0\}$. On the other hand,  for $s\in\HH$, $\Im
P_s(t(2+3i))=\Im[s-\left(\frac{46}{3}-3 i\right) t^3-(5-12 i) t^2]>0$ and
similarly, $\Im P_s(-2+t(-2+3i))>0$. Thus ${\ro}_2$ stays in  between the two rays $l_{1,2}$ for all $s\in\HH$.

(iii) Real analyticity in $t$ follows again from the IFT. Near $t=0$, the IFT
applied to the equations  $v\pm (-tv^3/3-s)^{1/2}=0$ at $r=\mp\sqrt{-s},t=0$
implies the existence of two roots of $tv^3/3+v^2+s$ analytic in $t$. When
$s$ is close to $-4/3$ the result follows from the symmetry $T$.

(iv) This simply follows from the fact that  for $|u|\Ge 399/100$ we have $|u^3/3+u^2|\Ge|(399/100)^3/3-(399/100)^2>21/4$. \end{proof}

\subsubsection{Estimates of sums of square roots}\label{EsSumsro}
Let $\eta_1$ be given by Lemma\,\ref{eps1} and $r_j$ as in \S\ref{roots_prop}.
By Lemma \ref{lem1} and Corollary \ref{lem2} (iii) we
have
$$\inf_{u\in\mathcal{C},\,s\in\newD6p\,{\rm{with}}\, |s|>\frac12\eta_1,|\sm|>\frac12\eta_1} \ |u-r_j(s)| \, >\, 0$$
and thus
\begin{equation}
  \label{eq:eta3}
  \eta_3:=\ \ \ \inf_{u\in\mathcal{C},\, s\in\newD6p \,{\rm{with}}\, |s|>\frac12\eta_1,|\sm|>\frac12\eta_1}\ |u^3/3+u^2+s|\, >\, 0.
\end{equation}

In the sequel, we will need estimates for  functions of the form
$\psi(u):=\psi(u;\sigma_{1,2,3})$ where
\begin{equation}\label{eq97}
\psi(u)= \sqrt{\sigma_1 u^3/3+u^2+s+\tilde{f}_{1}(u)}+\sigma_3\sqrt{\sigma_2 u^3/3+u^2+s+\tilde{f}_{2}(u)}=:\Phi_1+\Phi_2
\end{equation}
with $\sigma_{k}\in\{0,1\}$, $s+\tilde{f}_{1,2}\in \newD6p$. Here, $\tilde{f}_{1,2}$ are small
perturbations in the sense
\begin{equation}
  \label{eq:eqf12}
\sup_{u\in \mathcal{C}}|\tilde{f}_{1,2}(u)|\leqslant\eta_3/2\ \ \ \ { and } \ \ \  \ \sup_{u\in \oldCa}|s+\tilde{f}_{1,2}(u)|>|s|/10
\end{equation}

\begin{Note}\label{Def10} Upon analytic continuation in $s$ from $s\in(-4/3,0)$ to $s>0$ through $\HH$ the branches specified in Proposition\,\ref{js0}(i) give that $P_s(0)=s$ which for $s>0$ has zero argument, hence in $R(0,s)$ we choose the usual branch of
    the square root (which is positive when the argument is in $\RR^+$).
    \end{Note}

 In the following, $f$ is either $\tilde{f}_1$ or $\tilde{f}_2$.
   Since
    $f+s\in\HH$, by Lemma \ref{lem1}, none of the square roots vanishes
    on $\mathcal{C}$.  We analytically continue $\psi$ on $\mathcal{C}$ from $0$ to $u_0$\ \footnote{Recall that we use here $u_0=-4$, but the results are more general.}
clockwise and anticlockwise. This, of course, may result in a discontinuity
at $u_0$.
\begin{Lemma}\label{m0}
(i) If $\sigma_1=\sigma_2=\sigma_3=1$, then  $$\inf_{|s|\Ge\eta_1, |\sm|\Ge\eta_1,u\in\mathcal{C}}|\psi(u)|>0$$
(ii) If $|s|<\eta_1$ then, for all choices of $\sigma_i$ we
have
\begin{equation}\label{equmin}
\inf_{u\in\mathcal{C}}\left|\psi(u)(|u|^2+|s|)^{-1/2}\right|>0
\end{equation}
(iii) Let $\Phi_{10}$ and $\Phi_{20}$ the expressions defined in \eqref{eq97} with
$\sigma_1=\sigma_2=0$ and $\tilde{f}_i(u)$ replaced by $\tilde{f}_i(0)$. If in addition $\tilde{f}_k$ satisfy
 $|\tilde{f}_k(v)-\tilde{f}_k(0)|\lesssim |vs|+|s|^{3/2}$, then
$$\sup_{u\in\mathcal{C},|s|<\eta_1}\left|\frac{1}{\psi(u)}-
\frac{1}{\Phi_{10}+\Phi_{20}}\right|
\lesssim 1 $$

 (iv)  Let $\um,\sm$ as in \eqref{sym_T}. Similar statements hold for
$${\psi}^-(u):=\sqrt{\sigma_1
(\um)^3/3-(\um)^2+\sm+\tilde{f}_1(u)}+\sigma_3\sqrt{\sigma_2
(\um)^3/3-(\um)^2+\sm+\tilde{f}_2(u)}=:{\Phi}_{1}^-+{\Phi}_{2}^-$$
where $\tilde{f}_{1,2}$ satisfies \eqref{eq:eqf12} with $\oldCa$ replaced by $\oldCb$, and $s$ by $\sm$. To be precise we have

(ii') If $|\sm|<\eta_1$ then, for all choices of $\sigma_i$ we
have $$\inf_{u\in\mathcal{C}}\left|\psi^-(u)(|\um |^2+|\sm|)^{-1/2}\right|>0$$

(iii') Let ${\Phi}_{10}^-$ and ${\Phi}_{20}^-$ the expressions defined in \eqref{eq97} with
$\sigma_1=\sigma_2=0$ and $\tilde{f}_i(u)$ replaced by $\tilde{f}_i(-2)$. If in addition $\tilde{f}_k$ satisfy
 $|\tilde{f}_k(u)-\tilde{f}_k(-2)|\lesssim |\um\sm|+|\sm|^{3/2}$, then
$$\sup_{u\in\mathcal{C},|\sm|<\eta_1}\left|\frac{1}{\psi^-(u)}-
\frac{1}{{\Phi}_{10}^-+{\Phi}_{20}^-}\right|
\lesssim 1 $$

\end{Lemma}

\begin{proof}
(i) Note that  $\Phi_1^2/\Phi_2^2=1+\lambda(u)$ where
$\lambda(u)=(\tilde{f}_1-\tilde{f}_2)/\Phi_2^2$. By \eqref{eq:eta3} and \eqref{eq:eqf12} we have
$\sup_{u\in\mathcal{C}}|\lambda|=a_1<1$. By the choice of branches, see
Note \ref{Def10}, we have
$\Phi_1+\Phi_2=\Phi_2(1+\lambda)^{\frac12}$ and thus $\Phi_1+\Phi_2$ can only vanish if $\Phi_2$
does, and this is ruled out by \eqref{eq:eta3}.

(ii) For $u\in \oldCa$, by the choice of branch, $\Phi_{1,2}$ are in the
first quadrant. Then, $|\Phi_1+\Phi_2|\Ge \min\{|\Phi_1|,|\Phi_2|\}$, so we can reduce
the analysis to the case $\sigma_3=0$. If $\sigma_1=\sigma_3=0$, then the
estimate follows from the fact that $|u^2|+|s+f|\Le 2|u^2|+2|s|$. If
$\sigma_1=1$ and $\sigma_3=0$  the proof
is similar on $\oldCa$, where  $|u^3/3+u^2|=|u|^2|1+u/3|$ and $|u/3|<\tfrac4{25}$. On the rest of $\mathcal{C}$ we have, using Lemma \ref{lem1} and
Note\,\ref{Def10},  $\sqrt{\tfrac13u^3+u^2+s+f}$ and $\sqrt{u^2+s+f}$ are
the analytic continuations of $u\sqrt{1+\tfrac{u}3}\sqrt{1+g_1(u,s)}$ and
$u\sqrt{1+g_2(u,s)}$ respectively (these are
the branches when  $u\in\oldCa$ and large relative to $s$, and here
$g_1=-1+(1-\tfrac{3+{\ro}_1}{u+3})(1-\tfrac{{\ro}_3}{u})(1-\tfrac{{\ro}_2}{u})$), where $r_i$
are the roots of  $\tfrac13u^3+u^2+s+f$,
 $g_2=(1+\tfrac{s+f}{u^2})-1$,  $|g_{1,2}|<10\eta_1$
 and $|1+\tfrac{u}{3}|>\tfrac13$  and the estimate is
immediate.

(iii) This follows by straightforward estimates using (ii):
\begin{multline}\label{rv2}
\left|\frac{1}{\psi}-
\frac{1}{\Phi_{10}+\Phi_{20}}\right|\\
=\left|\frac{1}{\psi(\Phi_{10}+\Phi_{20})}\left(\frac{-\sigma_1v^3/3-(\tilde{f}_1(v)-\tilde{f}_1(0))}{
\Phi_{1}+\Phi_{10}}+\sigma_3\frac{-\sigma_2v^3/3-(\tilde{f}_2(v)-\tilde{f}_2(0))}{
\Phi_{2}+\Phi_{20}}\right) \right|\\
\lesssim \frac{|v|^3+|s_n v|}{\sqrt{(|v|^2+|s_n|)^3}}+\frac{|v|^3+|s_n|^{3/2}}{\sqrt{(|v|^2+|s_n|)^3}}\lesssim 1
\end{multline}
(iv) The proof is very similar: in fact, it follows by replacing $(u,s)$ by $(\um,\sm)$ and using the symmetry \eqref{sym_T} which is essentially inherited by the whole problem.
\end{proof}

\end{document}